\documentclass[11pt]{amsart}
\usepackage{amssymb,amsmath,mathabx,color,enumerate,enumitem,cancel,ifthen,hyperref,graphicx,ulem}
\usepackage{appendix}
\usepackage{comment}
\usepackage{kantlipsum}

\calclayout

\usepackage{xargs,soul}
\usepackage[pdftex,dvipsnames]{xcolor}
\usepackage[colorinlistoftodos,prependcaption,textsize=tiny]{todonotes}
\newcommandx{\unsure}[2]{\todo[linecolor=red,backgroundcolor=red!25,bordercolor=red]{#2}\texthl{#1}}

\newcommand{\Lip}{\ensuremath{\mathrm{Lip}}}

\newcommand{\Rea}{\mathbb{R}}
\newcommand{\Nat}{\mathbb{N}}
\newcommand{\Int}{\mathbb{Z}}
\newcommand{\F}{\mathcal{F}}
\newcommand{\HH}{\mathbb{H}^n}
\newcommand{\He}{\mathbb{H}}
\newcommand{\bG}{\mathbf{G}}

\newcommand{\geps}{g_{\varepsilon}}

\def\FF{\mathcal F}

\def\MM{\mathcal{M}}
\def\NN{\mathcal{N}}

\newcommand{\ver}[1][]{
   \ifthenelse{ \equal{#1}{} }
      {\ensuremath{V_{ext}}}
      {\ensuremath{V_{ext}(#1)}}
}
\newcommand{\edg}[1][]{
   \ifthenelse{ \equal{#1}{} }
      {\ensuremath{E_{ext}}}
      {\ensuremath{E_{ext}(#1)}}
}

\newcommand\absb[2]{\csname#1l\endcsname|#2\csname#1r\endcsname|}
\newcommand\norm[1]{\mathopen\|#1\mathclose\|}

\newcommand\normb[2]{\csname#1l\endcsname\|#2\csname#1r\endcsname\|}

\def \Span {\operatorname{span}}

\def \LIso {\operatorname{LIso}}

\def \ext {\operatorname{ext}} 
\def\restriction{\!\upharpoonright}
\def\setsep{:\;}
\newtheorem{theorem}{Theorem}[section]

\newtheorem{proposition}[theorem]{Proposition}
\newtheorem{corollary}[theorem]{Corollary}

\newtheorem{lemma}[theorem]{Lemma}
\newtheorem{fact}[theorem]{Fact}
\theoremstyle{definition}
\newtheorem{definition}[theorem]{Definition}
\newtheorem{notation}[theorem]{Notation}
\newcounter{claimcounter}[section]
\newtheorem{claim}[claimcounter]{Claim}
\newtheorem*{claim*}{Claim}

\newtheorem{remark}[theorem]{Remark}
\newtheorem{example}[theorem]{Example}
\newtheorem{examples}[theorem]{Examples}

\newcounter{que}
\newtheorem{question}[que]{Question}

\newtheorem*{example*}{Example}

\title{Isometries of Lipschitz-free Banach spaces}
\author[M. C\' uth]{Marek C\'uth}
\author[M. Doucha]{Michal Doucha}
\author[T. Titkos]{Tam\'as Titkos}
\email{cuth@karlin.mff.cuni.cz}
\email{doucha@math.cas.cz}
\email{titkos@renyi.hu}

\address[M.~C\' uth]{Charles University, Faculty of Mathematics and Physics, Department of Mathematical Analysis, Sokolovsk\'a 83, 186 75 Prague 8, Czech Republic}
\address[M.~Doucha]{Institute of Mathematics of the Czech Academy of Sciences, \v{Z}itn\'a 25, 115 67 Prague 1, Czech Republic}
\address[T.~Titkos]{HUN-REN Alfr\'ed R\'enyi Institute of Mathematics, Re\'altanoda u. 13-15.,
Budapest 1053, Hungary}


\thanks{M. C\'uth was supported by the GA\v{C}R project 23-04776S.
M. Doucha was supported by the GA\v{C}R project 22-07833K and by the Czech Academy of Sciences (RVO 67985840). T. Titkos is supported by the Hungarian National Research, Development and Innovation Office - NKFIH grant no. K134944) and by the Momentum program of the Hungarian Academy of Sciences under grant agreement no. LP2021-15/2021.}

\begin{document}
\maketitle
\begin{abstract}
 We describe surjective linear isometries and linear isometry groups of a large class of Lipschitz-free spaces that includes e.g. Lipschitz-free spaces over any graph. We define the notion of a Lipschitz-free rigid metric space whose Lipschitz-free space only admits surjective linear isometries coming from surjective dilations (i.e. rescaled isometries) of the metric space itself. We show this class of metric spaces is surprisingly rich and contains all $3$-connected graphs as well as geometric examples such as non-abelian Carnot groups with horizontally strictly convex norms. We prove that every metric space isometrically embeds into a Lipschitz-free rigid space that has only three more points.
\end{abstract}

\section{Introduction}

The study of linear isometries of Banach spaces is as old as Banach spaces themselves. Indeed, Banach already in \cite{Banachbook} described linear isometries of separable $C(K)$ spaces, a result that has later become known as the Banach-Stone theorem. Banach himself later described the isometries of $L_p[0,1]$, for $p\in \left[1,\infty\right)\setminus\{2\}$, which was later generalized by Lamperti to all $L_p(X,\mu)$ spaces and became known as the Banach-Lamperti theorem, which is relevant also for the topic of the present paper. Since then, linear isometries of many more classes of Banach spaces have been described and we refer the reader to the monographs \cite{FleJa1,FleJa2} for a detailed account on this topic. In this paper we continue this line of research for the class of Lipschitz-free Banach spaces and, to a smaller extent, also their duals. These spaces are free objects in the category of Banach spaces over the category of metric spaces and due to their close relation to Wasserstein spaces they attract the attention of researchers not only from (linear and non-linear) functional analysis, but also metric geometry, optimal transport and computer science, and have become one of the most intensively studied class of Banach spaces. We refer to \cite{Weaverbook} for an introduction.

Anyone familiar with these spaces will immediately notice that every isometry of the underlying metric space induces a linear isometry of the corresponding Lipschitz-free space. The converse is true for some metric spaces and fails dramatically for other. Understanding when and why this happens significantly improves our knowledge of the linear structure of these Banach spaces and we believe is worth of studying. Let us remark we are not the first ones to undertake this research. Surjective linear isometries of Lipschitz-free spaces over so-called uniformly concave metric spaces are fully described in \cite[Section 3.8]{Weaverbook}, which is based on a much older result of Mayer-Wolf from \cite{M-W81}. Very interesting results were obtained recently by Alexander, Fradelizi, Garc\' ia-Lirola and Zvavitch in \cite{AFGZ} where they describe linear isometries of finite-dimensional Lipschitz-free spaces using graph-theory techniques. One also has to keep in mind that Lipschitz-free spaces over $\Rea$ and $\Nat$ are $L_1$-spaces, therefore their isometries are described by the Banach-Lamperti theorem - not in the language of the Lipschitz-free spaces though, a work that is yet to be done and could shed some light on how to deal with isometries of Lipschitz-free spaces over some geodesic metric spaces such as $\Rea^n$.

This paper generalizes both \cite[Section 3.8]{Weaverbook} and \cite{AFGZ}, where especially the latter was an inspiration to many results presented here. Similarly as \cite{AFGZ}, it uses graph-theory techniques together with the recent advances of Aliaga et al. (\cite{AG19,AP20,APPP20}) on the structure of extreme points in Lipschitz-free spaces.\medskip

Let us summarize the main results of the paper.
\begin{itemize}
    \item For any metric space $\MM$, which can be viewed as a complete weighted directed graph, we define a certain subgraph $V(\MM)$. If $V(\MM)$ is connected and its vertices are dense in $\MM$, in which case $\MM$ is called a weak Prague space, there is a one-to-one correspondence between linear isometries of the Lipschitz-free space $\F(\MM)$ and special edge bijections of $V(\MM)$. See Section~\ref{sec:cyclicbijections} for details. In particular, we fully describe the linear isometry groups of Lipschitz-free spaces (and also $\Lip_0$-spaces) over any graphs. See Section~\ref{sec:graphsIsometries}.
    \item If the subgraph $V(\MM)$ is moreover $3$-connected, then $\MM$ is Lipschitz-free rigid meaning that there is a one-to-one correspondence between surjective linear isometries of $\F(\MM)$ and surjective dilations of $\MM$. See Theorem~\ref{thm:3ConnectedImpliesRigid}.
    \item We show that the class of Lipschitz-free rigid metric spaces is quite rich. Besides $3$-connected graphs it contains e.g. non-abelian Carnot groups with so-called horizontally strictly convex norms which includes the Heisenberg group with the Kor\' anyi norm (see Section~\ref{sec:examples}). One can conveniently build new Lipschitz-free rigid spaces using the $\ell_p$-sum operations (see Section~\ref{sec:buildingnew}). In particular, every metric space isometrically embeds into a Lipschitz-free rigid metric space containing only three new points (see Corrolary~\ref{cor:unionsRigid}).
\end{itemize}\medskip

We remark that there also exists a parallel and very active research on isometries and isometry groups of $p$-Wasserstein spaces of measures over metric spaces, initiated by Kloeckner in \cite{Kloe10} and \cite{BertrandKloeckner-Hadamard}. Because of this and the close relation between Wasserstein spaces and Lipschitz-free spaces mentioned above, one of our initial motivations was to explore the similarities and differences in the structure of isometries of Lipschitz-free spaces and $1$-Wasserstein spaces over the same metric space. It turns out that $1$-Wasserstein spaces seem to be much more rigid than Lipschitz-free spaces and the methods showing that are quite different. We refer to \cite{GTV-Hilbert} and references therein for a recent account on this topic.\medskip

The recommendation on how to read this paper is the following. The reader should read Sections~\ref{sec:prelim}, \ref{sec:cyclicbijections} and~\ref{sec:rigid} in their natural order, where we introduce the notations and basic facts, define the central notions of (weak) Prague spaces and characterize the surjective isometries of their Lipschitz-free spaces, and introduce the Lipschitz-free rigid metric spaces respectively. The remaining Sections~\ref{sec:buildingnew}, \ref{sec:graphsIsometries} and~\ref{sec:examples} can be read independently. There we build new Lipschitz-free rigid spaces, completely describe linear isometry groups of Lipschitz-free spaces over graphs, and present some examples, notably showing that non-abelian Carnot groups with horizontally strictly convex norms are Lipschitz-free rigid, respectively.

\section{Preliminaries and notation}\label{sec:prelim}

\subsection{Lipschitz-free spaces}
Let $(\MM,d)$ be a metric space. There are several equivalent definitions of the Lipschitz-free Banach space $\F(\MM)$ over $\MM$. The most common one requires a choice of a distinguished point $0\in\MM$ and $\F(\MM)$ is then defined as the closed linear span of $\{\delta(m)\colon m\in M\}\subseteq(\Lip_0(\MM))^*$, where $\delta(m):\Lip_0(\MM)\to\Rea$, for $m\in\MM$, is the evaluation functional $f\in\Lip_0(\MM)\to f(m)$. In this paper we shall use another definition (that is used in \cite{Weaverbook}) that avoids the necessity of choosing a point, whose advantage will become apparent shortly. Denote by $\F_0(\MM)$ the real vector space $\{x:\MM\to\Rea\colon |\mathrm{supp}(x)|<\infty,\; \sum_{m\in\MM} x(m)=0\}$ of \emph{molecules} over $\MM$. For any $x\in\F_0(\MM)$ and a function $f:\MM\to\Rea$ denote by $f(x)$ the sum $\sum_{m\in\MM} x(m)f(m)$ and set \[\|x\|:=\sup_f |f(x)|,\] where the supremum is taken over all $1$-Lipschitz functions. This is a norm on $\F_0(\MM)$ and we denote by $\F(\MM)$ the corresponding completion. The reason to use this definition is that given any bijective isometry $\phi:\MM\to\MM$, $\phi$ uniquely extends to a bijective isometry of $\F_0(\MM)$ and therefore also uniquely to a bijective isometry of $\F(\MM)$. In fact, given any surjective $a$-dilation $\phi:\MM\to\MM$, where $a>0$, $\phi$ uniquely extends to a bijective isometry of $\F(\MM)$. We recall that a map $\phi:\MM\to\MM$ is called an \emph{$a$-dilation} if for every $x,y\in\MM$ we have $d(x,y)=ad(\phi(x),\phi(y))$. Let us also note here that the group of linear isometries of a Banach space $X$ shall be denoted by $\LIso(X)$.

For $x,y\in\MM$, $x\neq y$ we denote $m_{x,y}:=\tfrac{\chi_{\{x\}} - \chi_{\{y\}}}{d(x,y)}\in \F(\MM)$, the \emph{normalized elementary molecule} supported on $x$ and $y$. An important feature of Lipschitz-free spaces is their universal property, see \cite[Theorem 3.6]{Weaverbook}. Namely, if we fix any base point $0_\MM\in\MM$, then the map $\delta:\MM\to \F(\MM)$ defined as $\delta(x) = d(x,0_\MM)m_{x,0_\MM}$ for $x\in\MM$ is an isometric embedding and moreover, for every Banach space $X$ and a Lipschitz map $f:\MM\to X$ with $f(0_\MM) = 0$ there is a unique bounded linear map $T_f:\F(\MM)\to X$ satisfying $\|T_f\| = \Lip(f)$ and $T_f\circ \delta = f$, we say that $T_f$ is the \emph{canonical linearization of $f$}.

We refer the reader to \cite[Section 3]{Weaverbook}, where all of the above (and more) information concerning Lipschitz-free spaces may be found.
\subsection{Graphs}
We shall work both with directed and undirected graphs in this paper. An \emph{undirected graph} is a pair $(V,E)$ where $V$ is a set (of vertices) and $E\subseteq \{e\in\mathcal{P}(V)\colon |e|=2\}$ is a set of edges viewed as a set of pairs of vertices. Notice that with this definition we do not allow loops and multiple edges between pairs of vertices.

A \emph{directed graph} is a pair $(V,E)$ where $V$ is again a set (of vertices) and $E\subseteq V^2\setminus \{(v,v)\colon v\in V\}$ is a set of ordered pairs of distinct vertices. We \emph{emphasize} that all the directed graphs $(V,E)$ in this paper have the property that for each directed edge $e=(v_1,v_2)\in E$ there is an edge $(v_2,v_1)\in E$ with the opposite orientation which we shall denote by $-e$. Notice that such directed graphs are in one-to-one correspondence with undirected graphs.

If $(V,E)$ is a directed graph and $e=(v_1,v_2)\in E$, then we call $v_1$, resp. $v_2$ the source, resp. range of $e$, and denote them by $s(e)$, resp. $r(e)$.   A \emph{cycle}, or more precisely a \emph{directed cycle}, in $(V,E)$ is a set of edges $\{e_1,\ldots,e_n\}\subseteq E$ such that $r(e_i)=s(e_{i+1})$, for $i<n$, and $r(e_n)=s(e_1)$. We say that the sequence of edges $(e_1,\ldots,e_n)\subseteq E$ is an \emph{unoriented cycle} if there are signs $(\varepsilon_i)_{i=1}^n\in \{-1,1\}^n$ such that $(\varepsilon_1 e_1,\ldots,\varepsilon_n e_n)$ is an oriented cycle. A cycle is \emph{simple} if it is not a concatenation of two proper subcycles. An \emph{edge path} from a vertex $v_1$ to a vertex $v_2$ is a sequence of edges $(e_1,\ldots,e_n)$ such that $s(e_1)=v_1$, $r(e_n)=v_2$ and $r(e_i)=s(e_{i+1})$ for $i<n$. There is also an obvious definition of a cycle for undirected graphs.

If $(V,E)$ is a directed graph and $A\subseteq E$, then we denote by $-A$ the set $\{-e\colon e\in A\}$ and say that $A$ is \emph{symmetric} if $A=-A$. For both a directed and undirected graph $(V,E)$ we shall often identify any $E'\subseteq E$ with the corresponding subgraph of $(V,E)$ whose set of edges is $E'$ and contains exactly those vertices that are incident to $E'$.

Occasionally, we also work with \emph{weighted graphs}, which is a graph (directed or undirected) equipped with a weight function $w:E\to\Rea_0^+$. If the graph is directed we also require that $w$ is symmetric as well, i.e. $w(e)=w(-e)$ for all $e\in E$.

A (directed or undirected) graph $(V,E)$ is \emph{$n$-connected}, for $n\geq 1$, if upon removing $n-1$ vertices the graph remains connected. Note that the one vertex graph is connected. We also note that directed graphs are connected if there is a \emph{directed} edge path between any two vertices. We define a relation $\sim$ on the set of edges $E$ of an undirected graph, where two edges are related if they lie on a common simple cycle. This relation is clearly symmetric and transitive. So, upon adding the diagonal to the relation, it is also reflexive and thus an equivalence relation. \emph{Edge components} are equivalence classes with respect to this equivalence relation. We extend these notions also for directed graphs $(V,E)$ that we consider, i.e. having the property that $E$ is symmetric. For each $e\in E$ let $[e]=\{e,-e\}$ and notice that $(V,\{[e]\colon e\in E\})$ is an undirected graph. Then set $e\sim f$, for $e,f\in E$ if $[e]\sim [f]$ in the undirected graph. The edge components of directed graphs are defined analogously.

We shall need the following fact concerning (infinite) graphs, which we will employ in a few situations. This is a standard theorem by Whitney for finite graphs. We refer to \cite[Section 3.2]{Graphbook} where the fact is proved. We remark that the implicit assumption there is that the graph is finite, but the proof works verbatim also in the infinite case.

\begin{fact}{\cite[Section 3.2]{Graphbook}}\label{fact:edgesConnected}Let $G = (V,E)$ be a graph. Then $G$ is 2-connected graph if and only if every two edges are contained in a simple cycle.
\end{fact}
\medskip

Let $(\MM,d)$ be a metric space. Put $E_M:=\MM^2\setminus\{(x,x)\colon x\in \MM\}$. For any $e=(x,y)\in E_M$ we denote $m_e=m_{x,y}$, and further we define a weight on $E_M$ by $w(e)=d(e):=d(x,y)$ (note that then we have $m_{-e}=-m_e$). Set also
\[
\edg[\MM] := \big\{(x,y)\setsep m_{x,y} \text{ is an extreme point in }B_{\FF(\MM)^{**}}\big\}.
\]
By $\ver[\MM]$ we denote the set $\{x\in\MM\colon \exists y\in \MM: (x,y)\in \edg[\MM]\}$ and by $G_{ext}(\MM)$ we denote the weighted directed graph $(\ver[\MM],\edg[\MM])$. If the metric space $\MM$ is clear from the context we write $\edg$, $\ver$ and $G_{ext}$ instead of $\edg[\MM]$, $\ver[\MM]$ and $G_{ext}(\MM)$ respectively.

Given a metric space $(\MM,d)$ and $x,y\in\MM$ we let $[x,y]:=\{z\in\MM\colon d(x,y) = d(x,z) + d(z,y)\}$ and if $m_{x,y}$ is an extreme point in $B_{\F(\MM)^{**}}$ we say it is a \emph{preserved extreme point of $B_{\F(\MM)}$}. 

In the following we collect several known results concerning (preserved) extreme points from the literature that is relevant for this paper.
\begin{fact}\label{fact:extremePoints}Let $\MM$ be a pointed metric space. Then all the preserved extreme points of $B_{\F(\MM)}$ are normalized elementary molecules. Moreover, given distinct points $x,y\in \MM$, we have
\begin{itemize}
    \item $m_{x,y}\in \ext B_{\F(\MM)}$ if and only if $[x,y] = \{x,y\}$,
    \item $m_{x,y}\in \ext B_{\F(\MM)^{**}}$ if and only if for every $\varepsilon>0$ there exists $\delta>0$ such that for every $z\in\MM$ with $\min\{d(x,z),d(y,z)\}\geq \varepsilon$ we have $d(x,y)\leq d(x,z) + d(y,z) - \delta$.
\end{itemize}
\end{fact}
\begin{proof}Preserved extreme points are normalized elementary molecules by \cite[Corollary 3.44]{Weaverbook}. The characterization of normalized elementary molecules which are preserved extreme points is contained in \cite[Theorem 4.1]{AG19}. Finally, the characterization of normalized elementary molecules which are extreme points is \cite[Theorem 3.2]{APPP20}.
\end{proof}
The information contained in Fact~\ref{fact:extremePoints} will be used in the sequel freely without a direct reference.

\section{Isometries on Lispchitz-free spaces generated by cyclic bijections}\label{sec:cyclicbijections}

In this section we introduce the fundamental class of (weak) Prague metric spaces and show that surjective linear isometries of their Lipschitz-free spaces are amenable to a description.
\begin{definition}
Let $(\MM,d)$ be a metric space. We say that $E\subset E_M$ is \emph{weakly admissible (for $\MM$)} if it is symmetric, $V_E:=\{x\in\MM\colon \exists y\in\MM: (x,y)\in E\}$ is dense subset of $\MM$ and the graph $G=(V_E,E)$ is connected. We say that $E$ is \emph{admissible} if moreover for every $x,y\in V_E$ we have
\begin{equation}\label{eq:pragueDistances}
    d(x,y) = \inf\Big\{\sum_{i=1}^n d(e_i)\colon e_1,\ldots,e_n \text{ is $E$-path from $x$ to }y\Big\}.
\end{equation}
\end{definition}

The connection of weakly admissible sets to isometries of Lipschitz-free spaces is explained in the following. In order to formulate our conditions in a short way we shall use the following notation.

\begin{notation}Let $(\MM_1,d_1)$ and $(\MM_2,d_2)$ be metric spaces with weakly admissible sets $E_1$ and $E_2$ and let $\sigma: E_1\to E_2$ be a bijection. In what follows we shall use the following conditions:
\begin{enumerate}[label=(S\alph*)]
    \item\label{it:wA1} $E'\subset E_1$ is a simple cycle if and only if $\sigma(E')$ is a simple cycle in $E_2$,
    \item\label{it:wA2} $\tfrac{d_1(e)}{d_2(\sigma(e))}$ is constant on each simple cycle $E'\subset E_1$,
    \item\label{it:wA3} for every $E_1$-path $e_1,\ldots,e_n$ from $x\in V_{E_1}$ to $y\in V_{E_1}$ and every $E_2$-path $f_1,\ldots,f_k$ from $x'\in V_{E_2}$ to $y'\in V_{E_2}$ we have
    \[
    d_1(x,y) = \normb{Big}{\sum_{i=1}^nd_1(e_i)m_{\sigma(e_i)}}\quad \text{and}\quad
    d_2(x',y') = \normb{Big}{\sum_{j=1}^kd_2(f_j)m_{\sigma^{-1}(f_j)}}.
    \]
\end{enumerate}
\end{notation}

\begin{proposition}\label{prop:canonicalIsometries}
Let $(\MM_1,d_1)$ and $(\MM_2,d_2)$ be metric spaces with weakly admissible sets $E_1$ and $E_2$. Let $\sigma: E_1\to E_2$ be a bijection and suppose that one of the following conditions hold
\begin{enumerate}[label=(\roman*)]
    \item\label{it:generalCond} $\sigma$ satisfies \ref{it:wA1}, \ref{it:wA2} and \ref{it:wA3},
    \item\label{it:admissibleCond} $E_1$ and $E_2$ are admissible and $\sigma$ satisfies \ref{it:wA1} and \ref{it:wA2}.
\end{enumerate}
Then there exists a surjective isometry $T:\F(\MM_1)\to \FF(\MM_2)$ such that $T(m_e)=m_{\sigma(e)}$, for every $e\in E_1$.

Moreover, \ref{it:admissibleCond} implies \ref{it:generalCond}.
\end{proposition}
\begin{proof}Assume that $\sigma$ satisfies \ref{it:wA1} and \ref{it:wA2}. First, we notice that the following holds
\begin{equation}\label{eq:simpleCycleImplication}
\text{$C\subset E_1$ is a cycle} \Rightarrow \sum_{e\in C}d_1(e)m_{\sigma(e)} = 0.
\end{equation}

Indeed, first consider the case when $C$ is a simple cycle. By \ref{it:wA2} there exists $K>0$ with $\tfrac{d_1(e)}{d_2(\sigma(e))}=K$ for $e\in C$ and so 
\[
\sum_{e\in C}d_1(e)m_{\sigma(e)} = K \sum_{e\in C}d_2(\sigma(e))m_{\sigma(e)} = K \sum_{e\in C} \big(\delta_{s(\sigma(e))} - \delta_{r(\sigma(e))}\big) = 0,
\]
where in the last equality we used that $\sigma(C)$ is a simple cycle by \ref{it:wA1}. Thus, \eqref{eq:simpleCycleImplication} holds whenever $C$ is a simple cycle. Now, given a cycle $C=(e_1,\ldots,e_n)$ we inductively find simple cycles $C_1,\ldots,C_k$ such that there is a permutation $\pi\in S_n$ and $i_1<\ldots<i_k\leq k$ with $C_j = (e_{\pi(i_j + 1)},\ldots,e_{\pi(i_{j+1})})$ for $j=1,\dots,k$: first we pick the first index $l$ such that there is $j<l$ with $r(e_j)=r(e_l)$ then put $C_1 := (e_{j+1},\ldots,e_l)$, and next we pick the cycle $(e_1,\ldots,e_j,e_{l+1},\ldots,e_n)$ and proceed inductively. Having simple cycles $C_1,\ldots,C_K$ as described above, we have 
\[
\sum_{e\in C}d_1(e)m_{\sigma(e)} = \sum_{i=1}^k \Big(\sum_{e\in C_i}d_1(e)m_{\sigma(e)}\Big) = 0
\]
where the last equality follows from the already proven part. Thus, we proved that \eqref{eq:simpleCycleImplication} holds for any cycle $C\subset E_1$.

Now, given $x,y\in V_{E_1}$ and any directed edge path $e_1,\ldots,e_n\in E_1$ from $x$ to $y$ we let
\[
\phi(x,y):=\sum_{i=1}^n d_1(e_i)m_{\sigma(e_i)}.
\]
By \eqref{eq:simpleCycleImplication}, this formula does not depend on the choice of the edge path because if $f_1,\ldots,f_k\in E_1$ is another edge path from $x$ to $y$, then  $e_1,\ldots,e_n,-f_m,\ldots,-f_1$ is a directed cycle and so we have 
\[\big(\sum_{i=1}^n d_1(e_i)m_{\sigma(e_i)}\big)-\big(\sum_{j=1}^m d_1(f_j)m_{\sigma(f_j)}\big) \stackrel{\eqref{eq:simpleCycleImplication}}{=} 0.\]

Now, we note that both \ref{it:admissibleCond} and \ref{it:wA3} imply that $\|\phi(x,y)\|\leq d_1(x,y)$ for any $x,y\in V_{E_1}$. Indeed, assuming \ref{it:wA3} we have even equality, so it suffices to check it under the assumptions that \ref{it:admissibleCond} holds. Pick $x\neq y\in V_{E_1}$. For every $\varepsilon>0$ there exists a directed edge path $e_1,\ldots,e_n\in E_1$ from $x$ to $y$ such that $d_1(x,y)\geq \big(\sum_{i=1}^n d_1(e_i)\big)-\varepsilon$. But then we obtain
\[\|\phi(x,y)\| = \normb{Big}{\sum_{i=1}^n d_1(e_i)m_{\sigma(e_i)}}\leq \sum_{i=1}^n d_1(e_i)\norm{m_{\sigma(e_i)}}=\sum_{i=1}^n d_1(e_i)\leq d_1(x,y)+\varepsilon\]
and since $\varepsilon>0$ was arbitrary, we obtain $\|\phi(x,y)\|\leq d_1(x,y)$.

Now, for $i\in \{1,2\}$ set $X_i:=\Span{\{m_e\colon e\in E_i\}}$. Since $E_i$ are weakly admissible, $X_i$ is a dense subspace of $\FF(\MM_i)$ for $i\in\{1,2\}$. We define $T_0:X_1\to X_2$ by setting \[T_0(m_e):=m_{\sigma(e)},\quad e\in E_1\] and extending it linearly. We shall verify that $T_0$ is a well-defined surjective linear isometry and thus it has a unique extension to a surjective linear isometry $T:\FF(\MM_1)\to\FF(\MM_2)$. 

Pick arbitrary base points $0_\MM\in V_{E_1}$ and $0_\NN\in V_{E_2}$. We claim there is a $1$-Lipschitz mapping $f:V_{E_1}\to \FF(V_{E_2})$ with $f(0_\MM) = 0$ such that $T_f|_{X_1} = T_0$, where $T_f: \F(V_{E_1})\to\F(V_{E_2})$ is the unique linear operator obtained from $f$ by the universal property of $\F(V_{E_1})$. Set $f(0_\MM):=0$. For $0_\MM\neq x\in V_{E_1}$ we set $f(x):=\phi(x,0_\MM)$. In order to check that $f$ is $1$-Lipschitz, pick $x\neq y\in V_{E_1}$ and a directed edge path $e_1,\ldots,e_n\in E_1$ from $x$ to $y$. Notice that $\delta_x-\delta_y=\sum_{i=1}^n d_1(e_i)m_{e_i}$, so
\[\begin{split}\|f(x)-f(y)\| & =\|T_0\big(\sum_{i=1}^n d_1(e_i)m_{e_i}\big)\|=\|\sum_{i=1}^n d_1(e_i)T_0(m_{e_i})\| = \|\phi(x,y)\|\leq d_1(x,y).
\end{split}\]
Thus, $f$ is $1$-Lipschitz. Finally, it is obvious from the way how we defined $f$ that $f(x) = T_0(\delta(x)-\delta(0_\MM))$ for $0_\MM\neq x\in V_{E_1}$, which implies that $T_{f}|_{X_1} = T_0$. This finishes the proof of the claim above.

Similarly, applying the above to $(T_0)^{-1}$, we obtain a $1$-Lipschitz mapping $g:V_{E_2}\to \FF(V_{E_1})$ with $g(0_\NN) = 0$ such that $T_g|_{X_2} = (T_0)^{-1}$, where $T_g:\F(V_{E_1})\to\F(V_{E_1})$ is obtained from $g$ by the universal property of $\F(V_{E_2})$. But then we observe that $T_f = (T_g)^{-1}$, so $T_f$ is linear isometry with inverse $T_g$ and since $\sigma$ is surjective, we finally deduce that $T_0 = T_f|_{X_1}$ is surjective linear isometry.

Notice that this implies that in fact for every $x,y\in V_{E_1}$ we have the equality $\phi(x,y)=d_1(x,y)$. Since, as we have shown above, $\phi(x,y)$ does not depend on the choice of the edge path between $x$ and $y$ (and noticing that similar arguments apply of course for edge paths between points from $V_{E_2}$) we conclude that \ref{it:admissibleCond} implies \ref{it:generalCond}, proving the `Moreover' part of the statement.
\end{proof}

Next we aim for a kind of converse to Proposition~\ref{prop:canonicalIsometries}. Notice that such a potential converse might have a chance to work only for isometries $T:\F(\MM)\to \F(\NN)$, where $\MM$ and $\NN$ are metric spaces with (weakly) admissible sets $E_1$ and $E_2$ respectively, only if $T$ maps $E_1$ to $E_2$. For this reason we need to work with subsets of $E_\MM$, resp. $E_\NN$ that are preserved by isometries. This is the content of Proposition~\ref{prop:bijectionEext} where we do not yet require that these subsets are (weakly) admissible.

In the proof we need the following variant of \cite[Lemma 2.5]{AFGZ}, the proof is analogical so we omit it.

\begin{lemma}\label{lem:linIndependent} Let $\MM$ be a metric space and $E'\subset E_M$ be a finite set such that for any $e\in E'$ it is not true that $-e\in E'$. Then $\{m_e\setsep e\in E'\}\subset \FF(\MM)$ is not linearly independent if and only if $E'$ contains an unoriented cycle.
\end{lemma}

\begin{proposition}\label{prop:bijectionEext}
Let $(\MM,d_\MM)$ and $(\NN,d_\NN)$ be metric spaces and $T:\F(\MM)\to \F(\NN)$ be a surjective isometry. Then there exists a symmetric bijection $\sigma:\edg[\MM]\to \edg[\NN]$, i.e. $\sigma(-e)=-\sigma(e)$ for $e\in \edg[\MM]$,
such that
\begin{itemize}
    \item $T(m_e) = m_{\sigma(e)}$, $e\in \edg[\MM]$,
    \item $\sigma$ satisfies \ref{it:wA1}, \ref{it:wA2} and \ref{it:wA3} with $E_1 = \edg[\MM]$ and $E_2 = \edg[\NN]$.
\end{itemize}
\end{proposition}
\begin{proof}Since $T$ is a linear isometry, it maps preserved extreme points to preserved extreme points. Since the only preserved extreme points are normalized elementary molecules, we obtain that for every $e\in \edg[\MM]$ there exists $\sigma(e)\in \edg[\NN]$ such that $T(m_e)= m_{\sigma(e)}$. Since $T$ is a bijection, we easily obtain that the mapping $\sigma:\edg[\MM]\to \edg[\NN]$ is a bijection as well. Further, note that for any $e\in \edg[\MM]$ we have $\sigma(-e)=-\sigma(e)$ which follows easily from the fact that $T(m_{-e}) = -T^*(m_e)$.

Let us pick a simple cycle $(e_1,\ldots,e_n)$ from $\edg[\MM]$. In the case that $n=2$, we have $\{e_1,e_2\} = \{e_1,-e_1\}$ and so $\{\sigma(e_1),\sigma(e_2)\} = \{\sigma(e_1),-\sigma(e_1)\}$. Consider now the case when $n\geq 3$. Then $\sigma(e_i)\neq -\sigma(e_j)$ for $i< j\leq n-1$ and
\[
\sum_{i=1}^{n-1} d(e_i)m_{\sigma(e_i)} = \big(\sum_{i=1}^{n-1} d(e_i)m_{e_i}\big) = T^*(0) = 0,
\]
so, by  Lemma~\ref{lem:linIndependent}, $\{\sigma(e_i)\setsep i\leq n-1\}$ contains an unoriented cycle, that is, there are $m\leq n$, $\varepsilon\in\{\pm 1\}^{m-1}$ and one-to-one mapping $\pi:\{1,\ldots,m-1\}\to\{1,\ldots,n-1\}$ such that $\{\varepsilon(i)\sigma(e_{\pi(i)})\setsep i\leq m-1\}$ is a cycle. Applying the above to $T^{-1}$ we obtain that $\{e_{\pi(1)},\ldots,e_{\pi(m-1)}\}$ contains an unoriented cycle and so $m=n$ and $\pi\in S(n-1)$ is a permutation. After using one more permutation we may without loss of generality assume that the source $f_i:=s\big(\varepsilon(i)\sigma(e_{\pi(i)})\big)$ is equal to the range of $\varepsilon(i-1)\sigma(e_{\pi(i-1)})$ for $i\leq n-1$ (where $\varepsilon(0):=\varepsilon(n-1)$ and $\pi(0):=\pi(n-1)$). Thus, we have

\[\begin{split}
0 & =\sum_{i=1}^{n-1} d(e_i)m_{\sigma(e_i)}=\sum_{i=1}^{n-1} d(e_{\pi(i)})m_{\sigma(e_{\pi(i)})} = \sum_{i=1}^{n-1} d(e_{\pi(i)})\varepsilon(i)m_{\varepsilon(i)\sigma(e_{\pi(i)})}\\
& = \sum_{i=1}^{n-1} \frac{d(e_{\pi(i)})}{d(\sigma(e_{\pi(i)}))}\varepsilon(i)\big(\delta_{f_i}-\delta_{f_{i+1}}\big)
\end{split}\]
and so, comparing the coefficients we obtain that $\tfrac{d(e_{\pi(i)})}{d(\sigma(e_{\pi(i)}))}\varepsilon(i)$, $i\leq n-1$ is constant.

Thus, $\varepsilon$ is constant and so, using the fact that $E'\subset E_{ext}$ is a cycle in $G_{ext}$ if and only if $-E'$ is a cycle in $G_{ext}$, we obtain that $\{\sigma(e_i)\setsep i\leq n-1\}$ is a simple cycle. Thus, $\sigma(E')$ is a simple cycle whenever $E'\subset E_{ext}$ is a simple cycle. Moreover, $\frac{d(e)}{d(\sigma(e))}$ is constant on each simple cycle $E'\subset E_{ext}$. This proves that $\sigma$ satisfies both \ref{it:wA1} and \ref{it:wA2}.

In order to check that \ref{it:wA3} holds, pick $\edg[\MM]$-path $e_1,\ldots,e_n$ from $x\in\MM$ to $y\in\MM$ and notice that then $\delta_\MM(x) - \delta_\NN(y) = \sum_{i=1}^n d_\MM(e_i)m_{e_i}$. Thus, since $T$ is isometry we obtain
\[
d_\MM(x,y) = \normb{Big}{\sum_{i=1}^n d_\MM(e_i)m_{e_i}} = \normb{Big}{\sum_{i=1}^n d_\MM(e_i)T(m_{e_i})} = \normb{Big}{\sum_{i=1}^n d_\MM(e_i)m_{\sigma(e_i)}}.
\]
Finally, for $\edg[\NN]$-path $f_1,\ldots,f_k$ from $x'\in\NN$ to $y'\in\NN$, using that $T^{-1}$ is isometry we similarly obtain that $d_\NN(x',y') = \normb{Big}{\sum_{i=1}^k d_\NN(f_j)m_{\sigma^{-1}(f_j)}}$.
\end{proof}

\begin{definition}
A metric space $\MM$ is called a \emph{Prague space} if $\edg[\MM]$ is admissible and a \emph{weak Prague space} if $\edg[\MM]$ is weakly admissible.\footnote{The name `Prague' was given to these spaces by the Budapest coauthor of this paper during his visit in Prague in order to simplify the wording.}
\end{definition}

Let us mention two important classes of Prague spaces. The first one is the following class from \cite{Weaverbook}.
\begin{definition}{\cite[Definition 3.3]{Weaverbook}}
A metric space $\MM$ is \emph{uniformly concave} if for every $x\neq y\in\MM$ and $\varepsilon>0$ there exists $\delta>0$ such that $d(x,z)+d(z,y)-\delta>d(x,y)$ for all $z\in\MM$ satisfying $\min\{d(x,z),d(y,z)\}\geq\varepsilon$.    
\end{definition}
Notice that it follows immediately from Fact~\ref{fact:extremePoints} that equivalently $\MM$ is uniformly concave if and only if $\edg[\MM]=\{(x,y)\colon x\neq y\in\MM\}$; in particular, such spaces are Prague.

The second class is the class of connected undirected graphs with the graph metric. We need the following simple lemma that immediately implies that every such a graph viewed as a metric space is a Prague space. We recall from the preliminaries that each undirected graph corresponds to a directed graph where each edge has its inverse. We may thus state the next lemma for such graphs.
\begin{lemma}\label{lem:Eext=E}
Let $(V,E)$ be a connected directed graph as in the paragraph above. Then $\edg=E$. In particular, $(V,E)$ with the graph metric is a Prague space.
\end{lemma}
\begin{proof}
    Given $x,y\in V$ we have that $[x,y] = \{x,y\}$ if and only if $(x,y)\in E$. Thus, $m_{x,y}\in \ext B_{\F(V)}$ only if $(x,y)\in E$, so we have $\edg\subset E$. Conversely, pick $(x,y)\in E$ and $\varepsilon>0$. Then for $\delta:=1$ we obtain that given $z\in V$ with $\min\{d(z,x),d(z,y)\}\geq \varepsilon$ we have
\[
d(x,z) + d(z,y) - d(x,y) \geq 2 - d(x,y) = \delta,
\]
so $(x,y)\in \edg$. Thus, $\edg = E$.
\end{proof}

We are now ready to state the main result of this section which directly follows from Propositions~\ref{prop:canonicalIsometries} and~\ref{prop:bijectionEext}.
\begin{theorem}\label{thm:1-1correspondencePrague}
Let $\MM$ and $\NN$ be weak Prague spaces.
\begin{enumerate}[label=(\roman*)]
    \item\label{it:weakPragueIsometries} There is one-to-one correspondence between surjective linear isometries between $\F(\MM)$ and $\F(\NN)$, and bijections between $\edg[\MM]$ and $\edg[\NN]$ satisfying  \ref{it:wA1}, \ref{it:wA2} and \ref{it:wA3}.
    \item\label{it:pragueIsometries} If both $\MM$ and $\NN$ are moreover Prague then there is a one-to-one correspondence between surjective linear isometries between $\F(\MM)$ and $\F(\NN)$, and bijections between $\edg[\MM]$ and $\edg[\NN]$  \ref{it:wA1} and \ref{it:wA2}.
\end{enumerate}
In both cases \ref{it:weakPragueIsometries} and \ref{it:pragueIsometries}, the correspondence is of the form $\sigma\to T_\sigma$, where $\sigma:\edg[\MM]\to\edg[\NN]$ is an appropriate bijection and $T_\sigma:\F(\MM)\to\F(\NN)$ is a linear isometry uniquely determined by the condition $T_\sigma(m_e)=m_{\sigma(e)}$ for all $e\in\edg[\MM]$.
\end{theorem}

\section{Lipschitz-free rigid spaces}\label{sec:rigid}

At this moment we are ready to investigate Lipschitz-free spaces whose surjective isometries are determined by surjective dilations of the underlying metric spaces. They will be the content of this section. A formal definition follows.

\begin{definition}
    Let $(\MM,d)$ be a metric space. We say it is \emph{Lipschitz-free rigid} if for any $T\in\LIso(\F(\MM))$ there exist $\varepsilon\in\{\pm1\}$, $a>0$ and a surjective $a$-dilation $g:\MM\to\MM$ such that $T(m_{x,y})=\varepsilon m_{g(x),g(y)}$ for every $x,y\in\MM$.
    Given a Lipchitz-free rigid space $\MM$, we say it is \emph{strongly Lipschitz-free rigid} if every isometry from $\LIso(\Lip_0(\MM))$ is an adjoint of some isometry from $\LIso(\F(\MM))$.
\end{definition}

\begin{remark}\label{rem:stronglyuniquepredual}We are not aware of any example of a metric space which is Lipschitz-free rigid but not strongly Lipschitz-free rigid. Note however that whenever $\F(\MM)$ is a strongly unique predual of $\Lip_0(\MM)$ then strong rigidity and rigidity coincide. It is not known whether $\F(\MM)$ is a strongly unique predual of $\Lip_0(\MM)$ for any metric space $\MM$, it is known that it is the case whenever $\MM$ is bounded or whenever $\MM$ is a geodesic space, see \cite[Section 3.4]{Weaverbook} for more details.
\end{remark}

\begin{remark}\label{rem:rigidGraphs}
We note that in \cite[Theorem 3.3]{Weaver18} (see also \cite[Theorem 3.27]{Weaverbook}) the author assumes that $\MM$ is a geodesic space, but in the proof he uses only that for any $x,y\in\MM$ and $n\in\Nat$ with $d(0,x)<n$ and $d(0,y)>n$ there exists $z\in\MM$ with $d(0,z) = n$ and $z\in [x,y]$; this condition is satisfied also for any connected graph endowed with the graph metric. So for any such graph $G$, $\F(G)$ is a strongly unique predual of $\Lip_0(G)$ and therefore $G$ is strongly Lipschitz-free rigid whenever it is Lipschitz-free rigid.
\end{remark}

\begin{proposition}\label{prop:rigidSufficient}
 Let $\MM$ be a weak Prague space. Suppose that $\edg$ is $2$-connected and whenever $\sigma:\edg\to\edg$ is a bijection satisfying \ref{it:wA1}, \ref{it:wA2} with $E_1=E_2=\edg$, there exist a graph isomorphism $f:\ver\to \ver$ and $\varepsilon\in \{\pm 1\}$ satisfying $\sigma(v,w)=\varepsilon(f(v),f(w))$ for every $(v,w)\in\edg$. Then $\MM$ is Lipschitz-free rigid. If $\MM$ is a Prague space, then the converse holds as well.
\end{proposition}
\begin{proof}
Let $\MM$ be a weak Prague space and suppose the condition from the statement is satisfied. Let $T:\F(\MM)\to\F(\MM)$ be a surjective linear isometry. By Theorem~\ref{thm:1-1correspondencePrague} there exists a bijection $\sigma:\edg\to\edg$ satisfying \ref{it:wA1} and \ref{it:wA2}. So by the assumption there are a graph isomorphism $f:\ver\to \ver$ and $\varepsilon\in \{\pm 1\}$ satisfying $\sigma(v,w)=\varepsilon(f(v),f(w))$ for every $(v,w)\in\edg$. such that $d(\sigma(e))/d(e)$ is constant on simple cycles. Since $\edg$ is $2$-connected and so by Fact~\ref{fact:edgesConnected} every two edges of $\edg$ lie on a common simple cycle, we get that $d(\sigma(e))/d(e)$ is globally constant. For every $(v,w)\in\edg$ we have $\frac{d(f(v),f(w))}{d(v,w)}=\frac{d(\sigma(v,w))}{d(v,w)}=:a$ and we claim that in general, for all $x\neq y\in\ver$ we have $d(f(x),f(y))=ad(x,y)$. Indeed, let $e_1,\ldots,e_n$ be an edge path in $\edg$ from $x$ to $y$. Then
\[d(f(x),f(y))=\big\|\sum_{i=1}^n d(\sigma(e_i))m_{\sigma(e_i)}\big\|=a\Big\|T\big(\sum_{i=1}^n d(e_i) m_{e_i}\big)\Big\| = a\big\|\sum_{i=1}^n d(e_i) m_{e_i}\big\|=a d(x,y).\] Thus $f$ extends to a surjective dilation $\bar f:\MM\to\MM$ which together with $\varepsilon$ witness that $\MM$ is Lipschitz-free rigid.\medskip

Now suppose that $\MM$ is a Lipschitz-free rigid Prague space. In order to get a contradiction suppose that $\edg$ is not $2$-connected. Then by Fact~\ref{fact:edgesConnected} $\edg$ contains at least two distinct edge components. Let $F\subseteq \edg$ be an edge component and $F':=\edg\setminus F$ which is non-empty. We define $\sigma:E\to E$ to be $\mathrm{Id}$ on $F$ and $-\mathrm{Id}$ on $F'$. Then clearly $\sigma$ satisfies \ref{it:wA1} and \ref{it:wA2} and so by Theorem~\ref{thm:1-1correspondencePrague} there exists a surjective isometry $T:\F(\MM)\to\F(\MM)$ such that $T(m_e)=m_{\sigma(e)}$ for every $e\in\edg$. Since $\MM$ is Lipschitz-free rigid there exist a surjective dilation $f:\MM\to\MM$ and $\varepsilon\in \{\pm 1\}$ such that $T(m_{x,y})=\varepsilon m_{f(x),f(y)}$ for all $x,y\in \ver$. Since $\edg$ is connected, there exist $x,y,z\in\ver$ such that $(x,y)\in F$ and $(x,z)\in F'$. Then $m_{x,y} = T(m_{x,y})=\varepsilon m_{f(x),f(y)}$ showing that $\varepsilon=+1$. On the other hand $m_{z,x} = T(m_{x,z})=\varepsilon m_{f(x),f(z)}$ showing that $\varepsilon=-1$, a contradiction.

Finally, suppose that $\sigma:\edg\to\edg$ is a bijection satisfying \ref{it:wA1} and \ref{it:wA2}. By Theorem~\ref{thm:1-1correspondencePrague} there exists a surjective linear isometry $T:\F(\MM)\to\F(\MM)$ satisfying $T(m_e)=m_{\sigma(e)}$ for every $e\in\edg$. Since $\MM$ is Lipschitz-free rigid, there exist $\varepsilon\in\{\pm1\}$, $a>0$ and a surjective $a$-dilation $g:\MM\to\MM$ such that $T(m_{x,y})=\varepsilon m_{g(x),g(y)}$ for every $x,y\in\MM$. Since $T$ maps $\{m_e\colon e\in\edg\}$ bijectively onto itself,  we have $(g(x),g(y))\in \edg$ if and only if $(x,y)\in\edg$ and therefore $g[\ver] = \ver$ and $g\restriction{\ver}$ is the required graph isomorphism.
\end{proof}

The converse for weak Prague spaces does not hold as the following example shows. Since the proof is quite long and technical, in order not to disturb the flow of the this section, this example can be found as Proposition~\ref{prop:rigidNot2Connected} in Section~\ref{sec:examples}.

\begin{example}\label{ex:rigidNot2Connected}
Consider the subset $I:=\{0\}\times [0,1] \subset \Rea^2$ and point $x_0:=(1,0)\in \Rea^2$. Then the metric space $(I\cup \{x_0\},\|\cdot\|_2)$ is a weak Prague space which is Lipschitz-free rigid, but $\edg$ is not $2$-connected.
\end{example}

It is a well-known result of Whitney that for $3$-connected finite graphs any cycle-preserving bijection between edges is induced by a graph isomorphism, see e.g. \cite[Section 5.3]{OMatroidBook}. So Proposition~\ref{prop:rigidSufficient} implies that whenever $\MM$ is finite and $\edg[\MM]$ is $3$-connected then $\MM$ is rigid. This was observed already in \cite{AFGZ}. We shall extend this observation to situations when $\edg$ is infinite, see Theorem~\ref{thm:3ConnectedImpliesRigid}. In order to prove it we need an analogy of the Whitney's theorem mentioned above for infinite graphs, see Corollary~\ref{cor:3ConnectedGraph}. We suspect it is well-known in graph theory, but we did not find any reference, so we include our proof below. In the remainder of this section, all the graphs are undirected. We recall that the directed graphs we consider, usually of the form $(\ver,\edg)$, are in one-to-one correspondence with undirected graphs.

\begin{definition}Let $G=(V,E)$ be a graph. We say $E'\subset E$ is a \emph{basis} of $E$ if it is a maximal set of edges which does not contain a simple cycle. Equivalently, it is a spanning forest, i.e. a disjoint union of trees that is incident with all the vertices of $V_{E'}$.\\
(Note that given connected graph $(V,E)$, there exists a basis $E_0\subset E$ which is then a spanning tree and then $(V,E_0)$ is again a connected graph.)
\end{definition}

\begin{lemma}\label{lem:verticesUsingCycles}
Let $G=(V,E)$ be a $2$-connected graph. For $v\in V$ put $E_v:=\{e\in E\colon v\notin e\}$ and let $E'\subsetneq E$. Then $E'\in \{E_v\colon v\in V\}$ if and only if the following two conditions hold:
\begin{enumerate}[label=(\roman*), series=basis]
    \item\label{it:2ConnectedOnEdges}
    $(\bigcup E',E')$ is connected;
    \item\label{it:condVertexBond}
for every $E_0\subset E'$ and $e\in E\setminus E'$, if $E_0$ is a basis of $E'$ then $E_0\cup \{e\}$ is a basis of $ E$.
\end{enumerate}

In particular, $E'\in \{E_v\colon v\in V\wedge E_v\text{ is $2$-connected}\}$ if and only if \ref{it:condVertexBond} holds together with the following condition
\begin{enumerate}[label=(\roman*), resume*=basis]
    \item\label{it:3ConnectedOnEdges} For every $e,e'\in E'$ there is a simple cycle $C\subset E'$ with $\{e,e'\}\subset C$.
\end{enumerate}
\end{lemma}
\begin{proof}First, pick $v\in V$ and let us verify that condition \ref{it:condVertexBond} holds for $E'=E_v$ (condition \ref{it:2ConnectedOnEdges} follows from $2$-connectedness of $G$). Pick a basis $E_0$ of $E_v$ and $e\in E\setminus E_v$. Since $E_0$ does not contain a simple cycle, $E_0\cup \{e\}$ does not contain a simple cycle as well. Moreover, since $G$ is $2$-connected, $E_v$ is connected and so $E_0$ is connected as well, which implies that $E_0\cup \{e\}$ is connected subset of $E$ and therefore it is a maximal set of edges not containing simple cycles, that is, it is a basis of $E$. Thus, \ref{it:condVertexBond} holds for $E'=E_v$.

On the other hand, pick some $E'\subsetneq E$ satisfying \ref{it:2ConnectedOnEdges} and \ref{it:condVertexBond}. First, we note that there exists $v\in V$ such that $E'\subset E_v$ as otherwise any basis $E_0$ of $E'$ is incident to all the vertices from $V$ and so for any $e\in E\setminus E'$ we have that $E_0\cup \{e\}$ contains a simple cycle, which contradicts \ref{it:condVertexBond}. Moreover, this $v$ is unique since if $E'\subseteq E_v\cap E_w$, for $v\neq w$, then for any spanning tree $E_0$ of $E'$ adding one $e\in E\setminus E'$ cannot be a spanning tree of $E$ because either $E_0\cup\{e\}$ is not a tree, or $E_0\cup\{e\}$ is incident just with one of the vertices $v$ and $w$. By \ref{it:2ConnectedOnEdges}, $(\bigcup E',E')$ is connected and so there exists $E_0\subset E'$ which is basis of $E'$ and $(\bigcup E',E_0)$ is connected. Thus, since $v$ is the unique vertex satisfying $E'\subset E_v$, $E_0$ is incident with all the vertices of $V$ except $v$. If there was $e\in E_v\setminus E'$, then by \ref{it:condVertexBond} we would obtain that $E_0\cup \{e\}$ is basis of $E$ which would however contradicts the fact that $E_0\cup \{e\}\subset E_v$ is not basis of $E$. Thus, we have $E_v\setminus E' = \emptyset$ and therefore $E' = E_v$.

The ``In particular'' part now easily follows using Fact~\ref{fact:edgesConnected} and the easy observation that \ref{it:3ConnectedOnEdges} implies \ref{it:2ConnectedOnEdges}.
\end{proof}

\begin{corollary}[Whitney's theorem for infinite graphs]\label{cor:3ConnectedGraph}Let $G=(V,E)$ be a $3$-connected graph and let $\sigma:E\to E$ be a simple cycle-preserving bijection. Then there is a bijection $f_\sigma:V\to V$ satisfying $\sigma(\{v,w\}) = \{f_\sigma(v),f_\sigma(w)\}$ for every $\{v,w\}\in E$.\\
(In particular, $f_\sigma$ is graph isomorphism.)
\end{corollary}
\begin{proof}For $v\in V$ put $E_v:=\{e\in E\colon v\notin e\}$. Since $G$ is $3$-connected, every $E_v$ is $2$-connected. Thus, by Lemma~\ref{lem:verticesUsingCycles}, because conditions \ref{it:3ConnectedOnEdges} and \ref{it:condVertexBond} are preserved by simple cycle-preserving bijections, there exists a bijection $f_\sigma:V\to V$ satisfying $\sigma(E_v) = E_{f_\sigma(v)}$ for any $v\in V$. Given $e = \{v,w\}$, we have $\sigma(e) = (E\setminus E_{f_\sigma(v)})\cap (E\setminus E_{f_\sigma(w)}) = \{f_\sigma(v),f_\sigma(w)\}$ and therefore $f_\sigma$ is graph isomorphism satisfying that $\sigma(\{v,w\}) = \{f_\sigma(v),f_\sigma(w)\}$ for every $\{v,w\}\in E$.
\end{proof}

  As an easy consequence of Proposition~\ref{prop:rigidSufficient} and Corollary~\ref{cor:3ConnectedGraph} we obtain the main result of this section.

\begin{theorem}\label{thm:3ConnectedImpliesRigid}
  Let $\MM$ be a weak Prague metric space such that $\edg[\MM]$ is $3$-connected. Then $\MM$ is Lipschitz-free rigid.
\end{theorem}
\begin{proof}Let $\sigma:\edg\to\edg$ be a symmetric bijection satisfying \ref{it:wA1},\ref{it:wA2}, for every $e\in\edg$. Let $(\ver,[\edg])$ be the undirected graph where $[\edg]:=\{[e]\colon e\in\edg\}$ and for each $e\in\edg$, $[e]=\{e,-e\}$. Let $[\sigma]:[\edg]\to[\edg]$ be the bijection defined by $[\sigma]([e])=[\sigma(e)]$. Clearly, it is simple cycle preserving and since $(\ver,[\edg])$ is $3$-connected, by Corollary~\ref{cor:3ConnectedGraph}, $[\sigma]$ is induced by a graph isomorphism $f:\ver\to\ver$. Let $(x,y)\in\edg$ be arbitrary and let $\varepsilon\in\{\pm 1\}$ be such that so that $\sigma(x,y)=\varepsilon(f(x),f(y))$ holds. We claim that $\varepsilon$ does not depend on the choice of $(x,y)\in\edg$. Otherwise, there are $(x,y),(x',y')\in\edg$ such that $\sigma(x,y)=(f(x),f(y))$ and $\sigma(x',y')=(f(y'),f(x'))$. Since $\edg$ is $2$-connected, there exists a simple directed cycle $e_1,\ldots,e_i,\ldots,e_n\in\edg$ such that $e_1=(x,y)$ and $e_i=(x',y')$. Let $j\leq n$ be the smallest index such that $(f(s(e_j)),f(r(e_j))=-\sigma(e_j)$. Notice that $1<j\leq i$. Then we get that $r(\sigma(e_{j-1}))\neq s(\sigma(e_j))$. This is a contradiction since $\sigma$ preserves simple directed cycles.

Consequently, we have $\sigma(x,y)=\varepsilon(f(x),f(y))$, for every $(x,y)\in\edg$. Since $\sigma$ was arbitrary, the assumption of Proposition~\ref{prop:rigidSufficient} is satisfied, so $\MM$ is Lipschitz-free rigid.
\end{proof}

\begin{remark}\label{rem:rigidNot3Connected}
We note that by Proposition~\ref{prop:rigidNot3Connected} from Section~\ref{sec:examples} there exists a Lipschitz-free rigid finite graph $G=(V,E)$ which is $2$-connected, but not $3$-connected. This shows that the converses to Whitney's theorem (Corollary~\ref{cor:3ConnectedGraph}) and Theorem~\ref{thm:3ConnectedImpliesRigid} do not hold even for finite graphs.
\end{remark}

The following is a basic example witnessing the first applications of Theorem~\ref{thm:3ConnectedImpliesRigid}. The first one is known (see \cite[Theorem 3.55]{Weaverbook}), the second is known only for finite graphs (see \cite[Corollary 4.3]{AFGZ}).

\begin{corollary}\label{cor:firstRigidExamples}
The following are examples of Lipschitz-free rigid Prague spaces.
\begin{itemize}
    \item Uniformly concave metric spaces.
    \item $3$-connected graphs $(V,E)$ endowed with the graph metric.
\end{itemize}
\end{corollary}

\begin{remark}
We note that the examples mentioned in Corollary~\ref{cor:firstRigidExamples} are even strongly Lipschitz-free rigid. For the case of uniformly concave metric spaces it follows from \cite[Theorem 3.56]{Weaverbook}, for the case of $3$-connected graphs we refer the reader to Remark~\ref{rem:rigidGraphs}.
\end{remark}

\section{Building new Lipschitz-free rigid spaces}\label{sec:buildingnew}

Although we already know examples of Lipschitz-free rigid metric spaces from Corrolary~\ref{cor:firstRigidExamples} and more will appear in Section~\ref{sec:examples}, it is also very useful to know that one can build new examples from the old ones using simple operations that are described below. The main application will be an isometric embedding of any metric space into a Lipschitz-free rigid space containing only three more points.
\subsection{Sums of metric spaces}\label{subsec:sums}

Given two metric spaces $\NN_1$ and $\NN_2$ and $p\in [1,\infty)$, we denote by $\NN_1\oplus_p \NN_2$ their $\ell_p$-sum, that is, the metric space $\NN_1\times\NN_2$ with metric defined as $d((x,y),(u,v)):=\norm{(d_{\NN_1}(x,u),d_{\NN_2}(y,v))}_p$, where $\|(r_1,r_2)\|_p$ denotes $\sqrt[p]{r_1^p+r_2^p}$. It is well-known that this defines a metric. One way how to check it is by the fact that $\NN_1\oplus_p \NN_2$ is canonically isometric to a subspace of the $\ell_p$-sum of Banach spaces $\FF(\NN_1)\oplus_p \F(\NN_2)$. Concerning extreme points we have the following.

\begin{proposition}\label{prop:pSumsExt}
    Let $\NN_1$, $\NN_2$ be pointed metric spaces, $p\in (1,\infty)$ and put $\MM:=\NN_1\oplus_p \NN_2$. Pick $(x_1,y_1),(x_2,y_2)\in\MM$ such that $m_{y_1,y_2}\in  B_{\F(\NN_2)}$ is a preserved extreme point. Then $m_{(x_1,y_1),(x_2,y_2)}\in B_{\F(\MM)}$ is a preserved extreme point as well.
\end{proposition}
\begin{proof}
 We shall use the characterization of preserved extreme points from Fact~\ref{fact:extremePoints}.

Pick $\varepsilon>0$. Then, we find $\eta\in (0,1)$ such that
\[
(1-\eta^p)^{1/p}>\eta\tfrac{d_{\NN_1}(x_1,x_2)}{d_{\NN_2}(y_1,y_2)}.
\]
By the assumption there exists $\delta_1>0$ such that given $z\in\NN_2$ with $\min\{d_{\NN_2}(y_1,z), d_{\NN_2}(y_2,z)\}\geq \eta\varepsilon$, we have $d_{\NN_2}(y_1,y_2) \leq  d_{\NN_2}(y_1,z) + d_{\NN_2}(y_2,z)- \delta_1$. Let 
 us denote by $\bullet$ the scalar product and consider the function $h:[0,\infty)^2\to \Rea$ defined by
\[
h(a,b):=(a,b)\bullet(d_{\NN_1}(x_1,x_2)^{p-1},d_{\NN_2}(y_1,y_2)^{p-1}) - \norm{(a,b)}_p\norm{(d_{\NN_1}(x_1,x_2)^{p-1},d_{\NN_2}(y_1,y_2)^{p-1})}_q.
\]
Note that by the H\"older inequality we have $h(a,b)\leq 0$ and $h(a,b)=0$ if and only if $ (a,b)$ is a multiple of $\big(d_{\NN_1}(x_1,x_2)^{q(p-1)/p},d_{\NN_2}(y_1,y_2)^{q(p-1)/p}\big) = \big(d_{\NN_1}(x_1,x_2),d_{\NN_2}(y_1,y_2)\big)$. 
Thus, we have $\delta_2:=-h(\varepsilon(1-\eta^p)^{1/p},\varepsilon\eta)>0$ and finally we put 
\[
\delta:=\frac{\min\{\delta_1\cdot d_{\NN_2}(y_1,y_2)^{p-1},\delta_2\}}{\norm{(d_{\NN_1}(x_1,x_2)^{p-1},d_{\NN_2}(y_1,y_2)^{p-1})}_q}.
\]
Pick $(x_3,y_3)\in\MM$ with
\begin{equation}\label{eq:presExtBigMin}
\min\Big\{d_{\MM}\big((x_1,y_1),(x_3,y_3)\big),d_{\MM}\big((x_3,y_3),(x_2,y_2)\big)\Big\}\geq \varepsilon.
\end{equation}
In order to shorten the notation we shall write $d_{i,j}$ and $\rho_{i,j}$ instead of $d_{\NN_1}(x_i,x_j)$ and $d_{\NN_2}(y_i,y_j)$ for $i,j\in \{1,2,3\}$, respectively.

Our aim is to show that we have
\begin{equation}
d_{\MM}\big((x_1,y_1),(x_2,y_2)\big)\leq d_{\MM}\big((x_1,y_1),(x_3,y_3)\big) + d_{\MM}\big((x_3,y_3),(x_2,y_2)\big) - \delta,
\end{equation}
which by the above mentioned characterization of preserved extreme points will imply that $m_{(x_1,y_1),(x_2,y_2)}$ is preserved extreme point.

We may without loss of generality assume that $\rho_{1,3}\leq \rho_{2,3}$. Thus, if $\rho_{1,3}\geq \eta\varepsilon$, by the choice of $\delta_1$ we obtain
\[\begin{split}
d_\MM\big( & (x_1,y_1),(x_2,y_2)\big) = \tfrac{1}{\|(d_{1,2},\rho_{1,2})\|_p^{p-1}}\Big\|\Big(d_{1,2},\rho_{1,2}\Big)\Big\|_p^p
\\
&\leq \tfrac{1}{\|(d_{1,2},\rho_{1,2})\|_p^{p-1}}\Big((d_{1,3} + d_{3,2})d_{1,2}^{p-1} + (\rho_{1,3} + \rho_{3,2}-\delta_1)\rho_{1,2}^{p-1}\Big)\\
& \stackrel{\text{H\"older}}{\leq} \tfrac{1}{\|(d_{1,2},\rho_{1,2})\|_p^{p-1}}\Big(\|(d_{1,3},\rho_{1,3})\|_p + \|(d_{3,2},\rho_{3,2})\|_p - \tfrac{\delta_1\cdot\rho_{1,2}^{p-1}}{\|(d_{1,2}^{p-1},\rho_{1,2}^{p-1})\|_q}\Big)\|(d_{1,2}^{p-1},\rho_{1,2}^{p-1})\|_q\\
& = \|(d_{1,3},\rho_{1,3})\|_p + \|(d_{3,2},\rho_{3,2})\|_p - \tfrac{\delta_1\cdot \rho_{1,2}^{p-1}}{\|(d_{1,2}^{p-1},\rho_{1,2}^{p-1})\|_q}\\
& \leq d_{\MM}((x_1,y_1),(x_3,y_3)) + d_{\MM}((x_3,y_3),(x_2,y_2)) - \delta.
\end{split}\]
Hence, it suffices to consider the case when $\rho_{1,3}<\varepsilon\eta$. By \eqref{eq:presExtBigMin} we have
\[
d_{1,3}\geq (\varepsilon^p - \rho_{1,3}^p)^{1/p} > \varepsilon(1-\eta^p)^{1/p}.
\]
The following estimate will help us to finish the proof.
\begin{claim*}
We have $\max\{h(a,b)\colon a\geq \varepsilon(1-\eta^p)^{1/p},\;b\leq \varepsilon\eta\} = -\delta_2$.
\end{claim*}
\begin{proof}[Proof of the Claim] Given $s,c,d\geq 0$, we consider the function
\[f_{s,c,d}(t):=(t,s)\bullet (c,d) - \norm{(t,s)}_p\cdot \norm{(c,d)}_q,\quad t\geq 0.\]
Then computing the derivative we find out that the function $f_{s,c,d}$ is increasing on $(0,s\tfrac{c^{q/p}}{d^{q/p}})$ and decreasing on $(s\tfrac{c^{q/p}}{d^{q/p}},\infty)$. Pick some $b\leq \varepsilon\eta$ and consider the function $f_{s_0,c_0,d_0}$ with $s_0=b$, $c_0=d_{\NN_1}(x_1,x_2)^{p-1}$ and $d_0=d_{\NN_2}(y_1,y_2)^{p-1}$. Note that then for $a\geq \varepsilon(1-\eta^p)^{1/p}$ we have 
\[
a\geq \varepsilon(1-\eta^p)^{1/p}>\eta\varepsilon\frac{d_{\NN_1}(x_1,x_2)}{d_{\NN_2}(y_1,y_2)} \geq b\frac{d_{\NN_1}(x_1,x_2)}{d_{\NN_2}(y_1,y_2)} =  s_0\frac{c_0^{q/p}}{d_0^{q/p}}
\]
and therefore $h(a,b) = f_{s_0,c_0,d_0}(a)\leq f_{s_0,c_0,d_0}(\varepsilon(1-\eta^p)^{1/p}) = h(\varepsilon(1-\eta^p)^{1/p},b)$. Similarly, consider the function $f_{s_1,c_1,d_1}$ with $s_1=\varepsilon(1-\eta^p)^{1/p}$, $c_1=d_{\NN_2}(y_1,y_2)^{p-1}$ and $d_1=d_{\NN_1}(x_1,x_2)^{p-1}$. Then for $b\leq \eta\varepsilon$ we have
\[
b\leq \eta\varepsilon < \varepsilon(1-\eta^p)^{1/p}\frac{d_{\NN_2}(y_1,y_2)}{d_{\NN_1}(x_1,x_2)} = s_1\frac{c_1^{q/p}}{d_1^{q/p}}
\]
and therefore $h(\varepsilon(1-\eta^p)^{1/p},b) = f_{s_1,c_1,d_1}(b)\leq f_{s_1,c_1,d_1}(\varepsilon\eta) = h(\varepsilon(1-\eta^p)^{1/p},\varepsilon\eta)$, which finishes the proof of the claim.
\end{proof}
Now, the following computation finishes the proof
\[\begin{split}
d_\MM\big( & (x_1,y_1),(x_2,y_2)\big) = \tfrac{1}{\|(d_{1,2},\rho_{1,2})\|_p^{p-1}}\Big\|\Big(d_{1,2},\rho_{1,2}\Big)\Big\|_p^p
\\
&\leq \tfrac{1}{\|(d_{1,2},\rho_{1,2})\|_p^{p-1}}\Big((d_{1,3} + d_{3,2})d_{1,2}^{p-1} + (\rho_{1,3} + \rho_{3,2})\rho_{1,2}^{p-1}\Big)\\
& = \tfrac{1}{\|(d_{1,2},\rho_{1,2})\|_p^{p-1}}\Big(h(d_{1,3},\rho_{1,3}) + \|(d_{1,3},\rho_{1,3})\|_p\|(d_{1,2}^{p-1},\rho_{1,2}^{p-1})\|_q + (d_{3,2},\rho_{3,2})\bullet(d_{1,2}^{p-1},\rho_{1,2}^{p-1})\Big)\\
& \stackrel{\text{H\"older}}{\leq} \tfrac{1}{\|(d_{1,2},\rho_{1,2})\|_p^{p-1}}\Big(\frac{h(d_{1,3},\rho_{1,3})}{\|(d_{1,2}^{p-1},\rho_{1,2}^{p-1})\|_q} + \norm{(d_{1,3},\rho_{1,3})}_p + \norm{(d_{3,2},\rho_{3,2})}_p\Big)\|(d_{1,2}^{p-1},\rho_{1,2}^{p-1})\|_q\\
& = \frac{h(d_{1,3},\rho_{1,3})}{\|(d_{1,2}^{p-1},\rho_{1,2}^{p-1})\|_q} + \norm{(d_{1,3},\rho_{1,3})}_p + \norm{(d_{3,2},\rho_{3,2})}_p\\
& \stackrel{Claim}{\leq} \|(d_{1,3},\rho_{1,3})\|_p + \|(d_{3,2},\rho_{3,2})\|_p - \tfrac{\delta_2}{\|(d_{1,2}^{p-1},\rho_{1,2}^{p-1})\|_q}\\
& \leq d_{\MM}((x_1,y_1),(x_2,y_2)) + d_{\MM}((x_2,y_2),(x_2,y_2)) - \delta.
\end{split}\]
\end{proof}

\begin{corollary}\label{cor:sums}
    Let $\NN_1$ and $\NN_2$ be metric spaces with $|\NN_1|\geq 3$ such that $\NN_2$ is a weak Prague space and $\ver[\NN_2]=\NN_2$. Let $p\in (1,\infty)$. Then for the metric space $\MM:=\NN_1\oplus_p\NN_2$ we have that $\ver[\MM] = \MM$ and $\edg[\MM]$is $3$-connected. In particular, $\MM$  is Lipschitz-free rigid.

    Moreover, if $\NN_2$ is uniformly concave and does not contain isolated points, then $\MM$ is a Prague space.
\end{corollary}
\begin{proof}
    By Proposition~\ref{prop:pSumsExt} we have
    \[
    \edg[\MM]\supset \{\big((x,y),(x',y')\big)\colon (y,y')\in \edg[\NN_2],\;x,x'\in\MM\}.
    \]
    Thus, by the assumption we have that $\ver[\MM] \supset \NN_1\times \ver[\NN_2] = \NN_1\times \NN_2$. Since $\ver[\NN_2]$ is connected, given points $(x,y),(x',y')\in\MM$ and an $\ver[\NN_2]$-path $z_0,\ldots,z_n$ from $y$ to $y'$, for every $z'\in\NN_1$ we have that $(x,z_0),(z',z_1),\ldots,(z',z_{n-1}),(x',z_n)$ is $\edg[\MM]$-path from $(x,y)$ to $(x',y')$ and therefore since $|\NN_1|\geq 3$, $\edg[\MM]$ is $3$-connected. The ``In particular'' part follows from Theorem~\ref{thm:3ConnectedImpliesRigid}.

    Finally, in order to check the ``Moreover'' part, assume additionally that $\NN_2$ is uniformly concave, i.e. that $\edg[\NN_2]=\{(x,y)\colon x\neq y\in\NN_2\}$, and that $\NN_2$ does not contain isolated points. Pick $(x,y),(x',y')\in\MM$. If $y\neq y'$ then we have $(x,y),(x,y')\in\edg[\MM]$ and there is nothing to check. So we may suppose that $y=y'$. Given $\varepsilon>0$, since $y$ is not isolated, we find $y_\varepsilon\in\NN_2$ with $0<d_{\NN_2}(y,y_\varepsilon)<\varepsilon$. Then by the assumption, for the $\edg[\MM]$-path $(x,y),(x,y_\varepsilon),(x',y)$ from $(x,y)$ to $(x',y')$ we obtain
   \[\begin{split}
   d_{\MM}\big((x,y),(x,y_\varepsilon)\big) & + d_{\MM}\big((x,y_\varepsilon),(x',y)\big) = d_{\NN_2}(y,y_\varepsilon) + \Big(d_{\NN_1}(x,x')^p + d_{\NN_2}(y,y_\varepsilon)^p\Big)^{1/p}\\
   & \leq \varepsilon + \Big(d_{\NN_1}(x,x')^p + \varepsilon^p\Big)^{1/p}\to d_{\NN_1}(x,x') =  d_{\MM}\big((x,y),(x',y')\big).
   \end{split}\] 
\end{proof}

\begin{corollary}\label{cor:sumsMain}
    Let $\MM$ be a metric space with $|\MM|\geq 3$ and let $\NN$ be a uniformly concave metric space. Then $\MM\oplus_p\NN$ is Lipchitz-free rigid for every $p\in(1,\infty)$.
\end{corollary}

\subsection{Disjoint unions of metric spaces}
For two sets $A$ and $B$ we shall denote their disjoint union by $A\sqcup B$. Given two bounded metric spaces, we easily obtain the following.
\begin{proposition}\label{prop:unionBdd}
    Let $\MM$ and $\NN$ be bounded metric spaces with $\min\{|\MM|,|\NN|\}\geq 3$. Then there exists a metric on $\MM\sqcup \NN$ extending metrics on $\MM$ and $\NN$ respectively, such that $\MM\sqcup \NN$ is Lipschitz-free rigid.
\end{proposition}
\begin{proof}We extend the metric by putting for $x\in \MM$ and $y\in \NN$, $d(x,y):=1 + \max\{\operatorname{diam} \MM, \operatorname{diam} \NN\}$, it follows from Fact~\ref{fact:extremePoints}, by setting $\delta=\varepsilon$, that $(x,y)\in \edg[\MM\sqcup\NN]$. Thus, we have
\[
\edg[\MM\sqcup\NN]\supset \{(x,y)\colon x\in\MM,\;y\in\NN\}
\]
and so $\ver[\MM\sqcup\NN] = \MM\sqcup \NN$ and since $\min\{|\MM|,|\NN|\}\geq 3$, $\edg[\MM\sqcup\NN]$ is $3$-connected. Thus, by Theorem~\ref{thm:3ConnectedImpliesRigid} $\MM\sqcup\NN$ is Lipschitz-free rigid.
\end{proof}

In particular, any bounded metric space with at least three points may be isometrically embedded into a Lipschitz-free rigid space with only three more points. As we shall see now, this holds even for unbounded metric spaces, see Corollary~\ref{cor:unionsRigid}.

\begin{proposition}\label{prop:unionsRigid}
    Let $\MM$ and $\NN$ be metric spaces such that $|\NN|\geq 4$ and $\NN$ is uniformly concave. Then for any point $0_\NN\in\NN$ there is a metric on $\MM\sqcup\NN\setminus\{0_\NN\}$ extending the metrics on $\MM$ and $\NN\setminus\{0_\NN\}$ respectively, such that $\MM\sqcup \NN\setminus\{0_\NN\}$ is Lipschitz-free rigid.
\end{proposition}
\begin{proof}
Fix any $p\in(1,\infty)$ and two points $0_\MM\in\MM$ and $0_\NN\in\NN$. Then there are isometric embeddings $\iota_\MM:\MM\to \MM\oplus_p\NN$, resp. $\iota_\NN:\NN\to\MM\oplus_p\NN$ defined by $x\in\MM\mapsto (x,0_\NN)\in \MM\oplus_p\NN$, resp. $y\in\NN\to (0_\MM,y)\in\MM\oplus_p\NN$. Identifying $\MM\sqcup\NN\setminus\{0_\NN\}$ with the disjoint union $\iota_\MM[\MM]\cup\iota_\NN[\NN\setminus\{0_\NN\}]\subseteq\MM\oplus_p\NN$ we obtain a metric on $\MM\sqcup\NN\setminus\{0_\NN\}$ inherited from $\MM\oplus_p\NN$ that extends the respective metrics on $\MM$, resp. $\NN\setminus\{0_\NN\}$. The extension, for $x\in\MM$ and $y\in\NN\setminus\{0_\NN\}$, can be explicitly given by the formula  
\[d(x,y)=\big(d_\MM^p(x,0_\MM) + d_{\NN'}^p(0_\NN,y)\big)^{1/p}.\]

By Proposition~\ref{prop:pSumsExt} and the assumption, we get for all $y\neq y'\in\NN\setminus\{0_\NN\}$ that $m_{\iota_\NN(y),\iota_\NN(y')}$ is a preserved extreme point in $B_{\F(\MM\oplus_p\NN)}$, so it is also a preserved extreme point in its subspace $B_{\F(\MM\sqcup \NN\setminus\{0_\NN\})}$. Analogously the same holds for every pair $x\in\MM$ and $y\in\NN\setminus\{0_\NN\}$.

Thus, we have $\ver[\MM\sqcup\NN\setminus\{0_\NN\}] = \MM\sqcup \NN\setminus\{0_\NN\}$ and $\edg[\MM\sqcup \NN\setminus\{0_\NN\}]$ is $3$-connected, because $|\NN\setminus\{0_\NN\}|\geq 3$. By Theorem~\ref{thm:3ConnectedImpliesRigid}, $\MM\sqcup\NN\setminus\{0_\NN\}$ is Lipschitz-free rigid.
\end{proof}

\begin{corollary}\label{cor:unionsRigid}Any metric space may be isometrically embedded into a Lipschitz-free rigid space with only three more points.
\end{corollary}

\section{Description of isometry groups of Lipschitz-free spaces over graphs}\label{sec:graphsIsometries}

Let us start this section with the following observation describing completely the isometry group $\LIso(\F(\MM))$ for Lipschitz-free rigid spaces $\MM$.

\begin{proposition}\label{prop:rigidLIso}
Let $\MM$ be a Lipschitz-free rigid metric space with $|\MM|\geq 3$. Let us denote by $\operatorname{Dil}(\MM)$ the group of all the dilations endowed with the topology of pointwise convergence. For each $g\in \operatorname{Dil}(\MM)$ let $T_g\in\LIso(\F(\MM))$ denote the unique isometry satisfying $T(m_{x,y}) = m_{g(x),g(y)}$ for $x,y\in\MM$, $x\neq y$. Then the mapping $\{-1,1\}\times \operatorname{Dil}(\MM)\ni (\varepsilon,g)\mapsto \varepsilon T_{g}\in \LIso(\F(\MM))$ is an isomorphism of topological groups.
\end{proposition}
\begin{proof}Since $\MM$ is Lipschitz-free rigid, the mapping is surjective. The fact that it is group homomorphism is easy and left to the reader. Further, in order to see the mapping is one-to-one, since we already know it is homomorphism it suffices to observe that whenever $\varepsilon T_g = Id$ then $\varepsilon = 1$ and $g = Id$. Indeed, we have $\varepsilon m_{g(x),g(y)} = m_{x,y}$ for every $x,y\in \MM$ with $x\neq y$ and so either $\varepsilon = 1$ in which case we obtain $g = Id$; or $\varepsilon = -1$ in which case we obtain $g(x) = y$ and $g(y) = x$ for every $x,y\in \MM$ with $x\neq y$ which is easily seen not to be possible whenever $\MM$ consists of at least three points.

Thus, we have checked that the mapping is isomorphism of groups and in order to check it is homeomorphism, it suffices to check that both the mapping and its inverse are continuous at the identity. In order to check continuity, we note that given a pointwise convergent net $(\varepsilon_i,g_i)\to (1,Id)$, we have eventually $\varepsilon_i = 1$ and $T_{g_i}(m_{x,y})\to m_{x,y}$ for every normalized elementary molecule $m_{x,y}$, which implies that $T_{g_i}\to Id$ pointwise (we are using the general and well-known fact that if a bounded net $S_i$ of operators between Banach spaces is pointwise convergent on a linearly dense set, then it is pointwise convergent on the whole domain as well).

Finally, in order to check continuity of the inverse, assume that there is a net $(\varepsilon_i,g_i)$ (where $g_i$ are $a_i$-dilations) satisfying $\varepsilon_i m_{g_i(x),g_i(y)} \to m_{x,y}$ for every $x,y\in \MM$ with $x\neq y$. In order to get a contradiction, assume that $\varepsilon_i$ is not eventually equal to $1$, so passing to a subnet we have $m_{g_i(x),g_i(y)} \to - m_{x,y} = m_{y,x}$, which using \cite[Lemma 2.2]{GLPPRZ18} implies that $g_i(x)\to y$ and $g_i(y)\to x$ for every $x,y\in\MM$ with $x\neq y$, which is however not possible whenever $\MM$ consists of at least three points, a contradiction. Thus, $\varepsilon_i$ is eventually equal to $1$ and so we have $m_{g_i(x),g_i(y)} \to m_{x,y}$ and using \cite[Lemma 2.2]{GLPPRZ18} we obtain $g_i\to Id$ pointwise and we are done.
\end{proof}

It is not true that every connected graph is Lipschitz-free rigid. Nevertheless, the linear isometry group of a Lipschitz-free space over any connected graph still admits precise description that will be provided in this section. Let us start with a general observation on the linear isometry groups of weak Prague spaces that shows that these groups are \emph{light} in the sense of \cite{Meg01}.

\begin{proposition}\label{prop:light}
Let $\MM$ be a weak Prague space. Then on the linear isometry group $\mathrm{LIso}(\F(\MM))$ the WOT topology coincides with the SOT topology.
\end{proposition}
\begin{proof}It is obvious that SOT convergence implies WOT convergence. Let us prove the converse. Let $(T_i)$ be a net in $\LIso(\F(\MM))$ which converges to $T\in \LIso(\F(\MM))$ in WOT topology. Let $\sigma_i$ and $\sigma$ be as given in Theorem~\ref{thm:1-1correspondencePrague}, that is, bijections on $\edg$ satisfying \ref{it:generalCond}, $T_i(m_e) = m_{\sigma_i(e)}$ and $T(m_e) = m_{\sigma(e)}$ for every $e\in\edg$. Then for every $e\in\edg$ we have that $m_{\sigma_i(e)} = T_i(m_e)\stackrel{w}{\to} T(m_e) = m_{\sigma(e)}$, but since weak and norm topologies coincide on molecules, see e.g. \cite[Lemma 2.2]{GLPPRZ18}, this implies that $T_i(m_e)\stackrel{\norm{\cdot}}{\to} T(m_e)$, $e\in\edg$. Since $\edg$ is weakly admissible the set $\{m_e\colon e\in\edg\}$ is linearly norm-dense in $\F(\MM)$, this implies that $T_i\to T$ in SOT topology.
\end{proof}

\begin{remark}In \cite{AFGR} the authors were interested in finding non/examples of \emph{light} Banach spaces, that is, Banach spaces $X$ for which SOT and WOT topology coincide on $\LIso(X)$. Proposition~\ref{prop:light} gives us an essentially new class of light Banach spaces. For example, since it is well-known that for $I=[0,1]$, $\F(I)$ is isometric to $L_1 = L_1(I)$, by Example~\ref{ex:rigidNot2Connected} we have that $L_1$ embeds isometrically as a hyperplane into a light Banach space (namely, into the Banach space $\F(I\cup \{x_0\})$, where $I\cup \{x_0\}$ is as in Example~\ref{ex:rigidNot2Connected}). This seems to be interesting since it is known that $L_1$ is not light, see \cite[Proposition 5.2]{AFGR}.
\end{remark}

\begin{proposition}\label{prop:SigmaLIso}
Let $\MM$ be a weak Prague (resp. Prague) metric space. Let $\Sigma$ be the topological group of all permutations of $\edg$ satisfying \ref{it:wA1}, \ref{it:wA2}, and \ref{it:wA3} (resp. \ref{it:wA1} and \ref{it:wA2}) with the topology of pointwise convergence, where $\edg$ has the topology inherited from $\F(\MM)$ (after identifying each $e\in \edg$ with $m_e\in \F(\MM)$). Then $\Sigma$ and $\LIso{\F(\MM)}$ are topologically isomorphic.
\end{proposition}
\begin{proof}
By Theorem~\ref{thm:1-1correspondencePrague}, the map $\sigma\in \Sigma\to T_\sigma\in\mathrm{LIso}(\F(\MM)$ is a bijection. It is straightforward that it is also a group homomorphism. So it suffices to check that it is a homeomorphism. Let $(\sigma_i)_{i\in I}\subseteq \Sigma$ be a net pointwise converging to $\sigma\in \Sigma$. Notice that for every $e\in\edg$ we then have $T_{\sigma_i}(m_e)$ converges in norm to $T_\sigma(m_e)$ and since $\{m_e\colon e\in\edg\}$ is linearly norm-dense in $\F(\MM)$, we get that $(T_{\sigma_i})_{i\in I}$ converges in SOT to $T_\sigma$. Conversely, let $(\sigma_i)_{i\in I}\subseteq\Sigma$ be a net such that $(T_{\sigma_i})_{i\in I}$ converges in SOT to $T_\sigma$. Fix $e\in \edg$. We have that $T_{\sigma_i}(m_e)$ converges in norm to $T_\sigma(m_e)$, thus by definition $\sigma_i(e)\to\sigma(e)$ in $\tau$, so $(\sigma_i)_{i\in I}$ converges to $\sigma$ pointwise.
\end{proof}

For the purpose of the following proposition we recall several notions. First, for a collection $(X_i)_{i\in I}$ of Banach spaces and $p\in[1,\infty]$ we denote by $(\bigoplus_{i\in I} X_i)_p$ the $\ell_p$-sum of Banach spaces $(X_i)_{i\in I}$. We shall also use the following.

\begin{definition}\label{def:wreath}
    Suppose we are given a sequence of topological groups $(G_i)_{i\in I}$ indexed by some set $I$ and a topological group $H$ acting by permutations on $I$. We shall write $h(i)$ for the action of $h\in H$ on $i\in I$. Assume that for each orbit $I'\subseteq I$ of $H$ on $I$  the groups $(G_r)_{r\in I'}$ are topologically isomorphic via some fixed isomorphisms $\phi_{r_1,r_2}:G_{r_1}\to G_{r_2}$ such that $\phi_{r_2,r_3}\circ\phi_{r_1,r_2}=\phi_{r_1,r_3}$, for $r_1,r_2,r_3\in I'$. We define the \emph{(unrestricted) wreath product $\prod_{i\in I} G_i\wr H$} to be the topological group which is as a topological space the direct product $\big(\prod_{i\in I} G_i\big)\times H$ and multiplication is defined by the rule $\big((g_i)_{i\in I},\pi\big)\cdot\big((h_i)_{i\in I},\pi'\big)=\big((g_i \phi_{\pi^{-1}(i),i}(h_{\pi^{-1}(i)}))_{i\in I},\pi\circ \pi'\big)$.\\
    (It is well-known and easy to check that such an operation really well-defines a group operation on the set $\big(\prod_{i\in I} G_i\big)\times H$.)
\end{definition}

\begin{remark}
Wreath products are sometimes defined under more stringent assumptions that the action of $H$ on $I$ is transitive and that all the groups $(G_i)_{i\in I}$ are equal instead of just being isomorphic; we refer e.g. to \cite[Chapter 8]{Wreathbook} for an introduction to wreath products of groups. It is clear that the second requirement is just a notational issue. For the first, notice that if $(I_j)_{j\in J}$ is an enumeration of the orbits of the action of $H$ on $I$ and we denote by $W_j$, for $j\in J$, the wreath product $\prod_{i\in I_j} G_i\wr H$, then the full wreath product $\prod_{i\in I} G_i\wr H$ is (topologically) isomorphic to the direct product $\prod_{j\in J} W_j$.
\end{remark}

\begin{proposition}\label{prop:wlog2ConnectedPrague}
    Let $\MM$ be a Prague space and $(E_i)_{i\in I}$ be the decomposition of $\edg[\MM]$ into edge components. Then the following holds.
    \begin{enumerate}[label=(\roman*)]
        \item\label{it:pragueComponents} For every $i\in I$ we have that $V_{E_i}:=\bigcup E_i\subset \MM$ is Prague space and $\edg[V_{E_i}]=E_i$ is $2$-connected;
        \item\label{it:isometricDecompositionIntoBiconnected} $\F(\MM)\equiv \Big(\bigoplus_{i\in I} \F(V_{E_i})\Big)_1$ isometrically;
        \item\label{it:wlogBiconnected} Let us denote by $S_I$ the group of all the bijections $\pi:I\to I$ such that for every $i\in I$ we have that $\F(V_{E_i})$ and $\F(V_{E_{\pi(i)}})$ are linearly isometric \footnote{For the canonical choice of mappings $\phi_{i,j}:\LIso(\F(V_{E_i}))\to \LIso(\F(V_{E_j}))$ for $i,j$ such that $\F(V_{E_i})$ and $\F(V_{E_j})$ are isometric we refer the reader to the proof below.}
        Then $\LIso(\F(M))$ is algebraically isomorphic to the wreath product $\prod_{i\in I}\LIso(\F(V_{E_i}))\wr S_I$.

        Moreover, if the topology on $\edg[\MM]\subset \F(\MM)$ is discrete, then this algebraic isomorphism is also a topological homeomorphism.
    \end{enumerate}
\end{proposition}

\begin{proof}\ref{it:pragueComponents}: Pick $i\in I$. It is easy to observe that $\edg[V_{E_i}] \supset E_i$, so we have $\ver[V_{E_i}] = V_{E_i}$ and $\edg[V_{E_i}]$ is $2$-connected. Moreover in the graph $\edg[\MM]$, for every simple edge path from $x\in V_{E_i}$ to $y\in V_{E_i}$ we have that all the edges in the path are from $E_i$ and since $\MM$ is Prague space, the condition \eqref{eq:pragueDistances} holds in the space $V_{E_i}$ and so $V_{E_i}$ is indeed a Prague space. It remains to prove that $\edg[V_{E_i}]\subset E_i$. Pick $(x,y)\in \edg[V_{E_i}]$. Since $E_i$ is an edge component, we have that $(x,y)\in \edg[\MM]$ implies $(x,y)\in E_i$. Thus, to conclude it suffices to prove that $(x,y)\in \edg[\MM]$. In order to get a contradiction, suppose this is not the case.

Then there exists $\varepsilon>0$ such that for every $\delta\in (0,\tfrac{\varepsilon}{2})$ there is $z_\delta\in \ver[\MM]\setminus V_{E_i}$ with $\min\{d(x,z_\delta),d(y,z_\delta)\} > \varepsilon$ and $d(x,y) >  d(x,z_\delta) + d(y,z_\delta)- \delta$. Since $z_\delta\notin V_{E_i}$, there exists a unique vertex $z_\delta'\in V_{E_i}$ such that for every $v\in V_{E_i}$ and every $\ver[\MM]$-path $x_0,\ldots,x_n$ from $z_\delta$ to $v$ we have $z_\delta'\in \{x_1,\ldots,x_n\}$. Thus, using that $\MM$ is Prague (we use condition \eqref{eq:pragueDistances}) we have $d(x,z_\delta) = d(x,z_\delta') + d(z_\delta',z_\delta)$ and $d(y,z_\delta) = d(y,z_\delta') + d(z_\delta',z_\delta)$. But then we obtain
\[\begin{split}
d(x,y) & > d(x,z_\delta) + d(y,z_\delta) - \delta =  d(x,z'_\delta) + d(y,z'_\delta) + 2d(z_\delta',z_\delta) - \delta\\
& > d(x,z'_\delta) + d(y,z'_\delta) - \delta.
\end{split}\]
 Moreover, by the above we also have
 \[
 d(z_\delta',z_\delta) < \frac{\delta + d(x,y) -  d(x,z'_\delta) - d(y,z'_\delta)}{2}\leq \frac{\delta}{2},
 \]
 which implies 
 \[
 \min\{d(x,z_\delta'),d(y,z_\delta')\} = \min\{d(x,z_\delta),d(y,z_\delta)\} - d(z_\delta',z_\delta)\geq \varepsilon - \tfrac{\delta}{2} \geq \tfrac{\varepsilon}{2}.
 \]
But since $\delta>0$ was arbitrary and $z_\delta'\in V_{E_i}$, this contradicts the fact that $(x,y)\in \edg[V_{E_i}]$.

\medskip

\noindent\ref{it:isometricDecompositionIntoBiconnected}: Put $E_I:=\{\{i,j\}\colon i,j\in I,\;i\neq j,\; V_{E_i}\cap V_{E_j}\neq \emptyset\}$ and note that the graph $T=(I,E_I)$ is a tree, that is, for any $i,j\in I$, $i\neq j$ there is a unique $E_I$-path from $i$ to $j$. Moreover, for every $\{i,j\}\in E_I$ there exists a unique vertex $v_{i,j}\in V_{E_i}\cap V_{E_j}$. Thus, for any $x\in V_{E_i}$ and $y\in V_{E_j}$ and the unique $E_I$-path $i_1,\ldots,i_n$ from $i$ to $j$ we have that any $\edg[\MM]$-path from $x$ to $y$ passes through vertices $v_{i_{1},i_2},v_{i_2,i_3},\ldots,v_{i_{n-1},i_n}$, and using that $\MM$ is Prague (we use condition \eqref{eq:pragueDistances}) we obtain that
\[d(x,y) = d(x,v_{i_1}) + \sum_{k=1}^{n-1} d(v_{i_{k}},v_{i_{k+1}}) + d(v_{i_n},y).\]
and for any $f\in\Lip_0(\ver[\MM])$
\[\begin{split}
|f(x)-f(y)| & \leq |f(x)-f(v_{i_1})| + \sum_{k=1}^{n-1} |f(v_{i_{k}}) - f(v_{i_{k+1}})| + |f(v_{i_n})-f(y)|\\
& \leq \sup_{i\in I} \norm{f|_{V_{E_i}}}_{\Lip}\Big(d(x,v_{i_1}) + \sum_{k=1}^{n-1} d(v_{i_{k}},v_{i_{k+1}}) + d(v_{i_n},y)\Big)\\&  = \sup_{i\in I} \norm{f|_{V_{E_i}}}_{\Lip} d(x,y),
\end{split}\]
so we have that $\norm{f}_{\Lip} = \sup_{i\in I} \norm{f|_{V_{E_i}}}_{\Lip}$ and the mapping $\Lip_0(\ver[\MM])\ni f\mapsto \Phi(f):=(f|_{V_{E_i}} - f(0_{V_{E_i}}))_{i\in I}\in \Big(\bigoplus_{i\in I} \Lip_0(V_{E_i})\Big)_\infty$ is an isometry, which is $w^*$-$w^*$ continuous, which follows from the Banach-Diedonn\'e theorem together with the fact that whenever $f_\gamma\stackrel{w^*}{\to} f$ is a bounded net, then $f_\gamma\to f$ pointwise and therefore for every $i\in I$ we obtain that $f_\gamma|_{V_{E_i}} - f_\gamma(0_{V_{E_i}})\to f|_{V_{E_i}} - f(0_{V_{E_i}})$ pointwise (equivalently, in the $w^*$-topology of $\Lip_0(V_{E_i})$). Moreover, this isometry $\Phi$ is surjective.

Indeed, pick $i_0$ with $0_\MM\in V_{E_{i_0}}$, put $c_{i_0}:=0$ and, for any $i\in I\setminus\{i_0\}$ we let \[c_i:=\sum_{k=1}^n (f_{i_{k-1}}(v_{i_{k-1},i_k})-f_{i_k}(v_{i_{k-1},i_k})),\] where $i_0,\ldots,i_n$ is the $E_I$-path from $i_0$ to $i$. Given $(f_i)_{i\in I}\in \Big(\bigoplus_{i\in I} \Lip_0(V_{E_i})\Big)_\infty$, we define $f:\MM\to \Rea$ by  $f(x):=f_i(x) + c_i$ for $x\in V_{E_{i}}$. This is a well-defined mapping because for $x\in V_{E_{i}}\cap V_{E_j}$ we have $x = v_{i,j}$, so we may without loss of generality assume that the $E_I$-path from $i_0$ to $i$ and $j$ is $i_0,\ldots,i_m,i$ and $i_0,\ldots,i_m,i,j$ respectively, which implies that $c_j = c_i + f_{i}(v_{i,j})-f_{j}(v_{i,j})$ and so
\[f_i(v_{i,j}) + c_i = f_i(v_{i,j}) + c_j + f_{j}(v_{i,j}) - f_{i}(v_{i,j}) = f_j(v_{i,j}) + c_j.\]
Since $f|_{V_{E_i}}$ is Lipschitz for every $i\in I$, by the above we have that $f\in \Lip_0(\MM)$ and moreover $\Phi(f) = (f_i)_{i\in I}$, because for $i\in I$ and $x\in V_{E_i}$ we have $f(x) - f(0_{V_{E_i}}) = f_i(x) - f_i(0_{V_{E_i}}) = f_i(x)$.

Thus, the mapping $\Phi$ defined above is adjoint of an isometry between $\Big(\bigoplus_{i\in I} \F(V_{E_i})\Big)_1$ and $\F(\ver[\MM])$, which using that $\ver[\MM]$ is dense in $\MM$ proves \ref{it:isometricDecompositionIntoBiconnected}.

\medskip

\noindent\ref{it:wlogBiconnected}: Let $\mathcal{E}$ be the decomposition of $I$ into equivalence class where $i,j\in I$ are equivalent if $\F(V_{E_i})\equiv \F(V_{E_j})$. Choose a representative $i_E\in E\subseteq I$ for each $E\in\mathcal{E}$ and for each $j\in E$ fix a linear isometry $T_{i_E,j}:\F(V_{i_E})\to \F(V_j)$. Then for every $E\in\mathcal{E}$ and $i,j\in E$ set $T_{i,j}$ to be the identity if $i=j$ and otherwise set $T_{i,j}:=T_{i_E,j}\circ T^{-1}_{i_E,i}$. Notice that $T_{i,k}=T_{j,k}\circ T_{i,j}$ for all $i,j,k\in E\in\mathcal{E}$. Finally, for each $i,j\in E\in\mathcal{E}$ we define $\phi_{i,j}:\LIso(\F(V_i))\to\LIso(\F(V_j))$ by $g\in\LIso(\F(V_i))\mapsto T_{i,j}\circ g\circ T_{j,i}$.

Let us denote by $\Sigma$, as in Proposition~\ref{prop:SigmaLIso}, the topological group of all bijections $\sigma:\edg[\MM]\to\edg[\MM]$ satisfying conditions \ref{it:wA1} and \ref{it:wA2}. By Proposition~\ref{prop:SigmaLIso}, we have a topological group isomorphism between $\Sigma$ and $\LIso(\F(\MM))$ given by the map $\sigma\in\Sigma\to T_\sigma\in\LIso(\F(\MM))$. Note that given $\sigma\in\Sigma$ there exists $\pi\in S_I$ such that we have $\sigma(E_i)=E_{\pi(i)}$, $i\in I$ and so we may define the mapping $\Omega:\Sigma\to \prod_{i\in I} \LIso\big(\F(V_{E_i})\big)\wr S_I$ by \[
\Omega(\sigma):=\Big(\big(T_{\sigma|_{E_{\pi^{-1}(i)}}}\circ T_{i,\pi^{-1}(i)}\big)_{i\in I},\pi\Big).
\]
We easily observe that $\Omega$ is one-to-one. Further, $\Omega$ is surjective because given $((T_i)_{i\in I},\pi)\in \prod_{i\in I} \LIso\big(\F(V_{E_i})\big)\wr S_I$ we have that $T_{\pi(i)}\circ (T_{i,\pi(i)}): \F(V_{E_i})\to \F(V_{E_{\pi(i)}})$ is linear isometry and so by the already proved part \ref{it:pragueComponents} and by Proposition~\ref{thm:1-1correspondencePrague} we find bijections $\sigma_i:E_i\to E_{\pi(i)}$, $i\in I$ satisfying \ref{it:wA1} and \ref{it:wA2} with $T_{\sigma_i} = T_{\pi(i)}\circ (T_{i,\pi(i)})$ and now we easily observe that for $\sigma:=\bigcup_{i\in I} \sigma_i:E\to E$ we have $\Omega(\sigma) = ((T_i)_{i\in I},\pi)$. Thus $\Omega$ is a bijection.

Let us check that $\Omega$ is a homomorphism. Fix $\sigma_1,\sigma_2\in\Sigma$. We have
\[\begin{split}
\Omega(\sigma_1)\Omega(\sigma_2) &=\Big(\big((T_{\sigma_1\upharpoonright E_{\pi_1^{-1}(i)}}\circ T_{i,\pi_1^{-1}(i)}\big)_{i\in I},\pi_1\Big) \Big(\big((T_{\sigma_2\upharpoonright E_{\pi_2^{-1}(i)}}\circ T_{i,\pi_2^{-1}(i)}\big)_{i\in I},\pi_2\Big)\\
&= \Big(\big((T_{\sigma_1\upharpoonright E_{\pi_1^{-1}(i)}}\circ T_{i,\pi_1^{-1}(i)}\circ \phi_{\pi_1^{-1}(i),i}(T_{\sigma_2 \upharpoonright E_{\pi_2^{-1}\pi_1^{-1}(i)}}\circ T_{\pi_1^{-1}(i),\pi_2^{-1}\pi_1^{-1}(i)})\big)_{i\in I},\pi_1\pi_2\Big)\\
& = \Big(\big(T_{\sigma_1\upharpoonright E_{\pi_1^{-1}(i)}}\circ T_{i,\pi_1^{-1}(i)}\circ T_{\pi_1^{-1}(i),i}\circ T_{\sigma_2 \upharpoonright E_{\pi_2^{-1}\pi_1^{-1}(i)}}\circ T_{\pi_1^{-1}(i),\pi_2^{-1}\pi_1^{-1}(i)}\circ T_{i,\pi_1^{-1}(i)}\big)_{i\in I},\pi_1\pi_2\Big)\\
&=\Big(\big(T_{\sigma_1\upharpoonright E_{\pi_1^{-1}(i)}}\circ T_{\sigma_2 \upharpoonright E_{\pi_2^{-1}\pi_1^{-1}(i)}}\circ T_{i,\pi_2^{-1}\pi_1^{-1}(i)}\big),\pi_1\pi_2\Big)\\ &= \Big(\big(T_{\sigma_1\sigma_2 \upharpoonright E_{\pi_2^{-1}\pi_1^{-1}(i)}} \circ T_{i,\pi_2^{-1}\pi_1^{-1}(i)}\big),\pi_1\pi_2\Big)\\ &=\Omega(\sigma_1\sigma_2).
\end{split}\]

Finally, if $\edg[\MM]$ is discrete topological space then for any convergent net $(\sigma_\gamma)$ in $\Sigma$ we have that the corresponding permutations $(\pi_\gamma)$ are pointwise eventually constant and using this observation it is easy to observe that then $\Omega$ is even a topological homeomorphism.
\end{proof}

In the rest of this section we shall completely describe the group of linear isometries of a Lipschitz-free space over an arbitrary undirected graph viewed as a metric space with graph metric. We note that this description could be extended to weighted graphs or even to more general Prague metric spaces, however notationally it would become cumbersome, so for simplicity we stick to the graph case and leave the possible generalizations to the reader.\medskip

Let $(M,E)$ be an undirected graph. Our aim is to describe the linear isometry group of $\FF((M,E))$ (or more simply just $\F(M)$). Since there is a bijection between undirected graphs and directed graphs, where for each edge there is also its inverse, we shall in fact assume that $M$ is directed and for every $e\in E$ there is also $-e\in E$, its inverse. Next we claim that we may without loss of generality assume that $M$ is $2$-connected. Indeed, this follows from Proposition~\ref{prop:wlog2ConnectedPrague}, Lemma~\ref{lem:Eext=E} and the next lemma.

\begin{lemma}\label{lem:graphs}
Let $(M,E)$ be a directed graph as in the paragraph above. Then $\edg[M]$ is a discrete subset of $\F(M)$.
\end{lemma}
\begin{proof}
First recall that $\edg[M]=E$ by Lemma~\ref{lem:Eext=E}. Now for two distinct $(x,y),(x',y')\in E$, using \cite[Lemma 1.2]{V23} and supposing that $\|m_{x,y} - m_{x',y'}\| < 2$ we have
\[\begin{split}
\|m_{x,y} - m_{x',y'}\| & = \|m_{x,y} + m_{y',x'}\| = \frac{d(x,x') + d(y,y') + |d(x,y)-d(x',y')|}{\min\{d(x,y),d(x',y')\}}\\
& = d(x,x') + d(y,y')\geq 1,
\end{split}\]
so when we identify $\edg[M]$ with a subset of $\F(M)$ we obtain that $\edg[M]$ is $1$-separated and therefore discrete.

\end{proof}

Thus, from now on we assume that our graph $(M,E)$ is $2$-connected. By Proposition~\ref{prop:bijectionEext}, every linear bijective isometry $T:\F(M)\to\F(M)$ induces a symmetric bijection $\sigma:E\to E$ satisfying \ref{it:wA1} (notice that the condition \ref{it:wA2} is trivial in the case of graphs; in fact $d(e)=d(\sigma(e))=1$ for all $e\in E$). Let $B$ be the Boolean algebra generated by simple directed cycles consisting of at least 3 edges. $\sigma$ induces also a Boolean algebra isomorphism $\pi:B\to B$. Call the atoms of $B$ \emph{pieces}. That is, each piece is an intersection of simple directed cycles and $\pi$ maps pieces to pieces. Moreover, for each piece $p$ we also have its inverse $-p$ and $\pi(-p)=-\pi(p)$.\medskip

We define a bipartite graph $(V,F)$ where $V=V_1\coprod V_2$ and a labelling $\lambda$ of $V_1$ by natural numbers. The set of vertices $V$ is a disjoint union $V_1\coprod V_2$, where $V_1$ is the set of all (directed) pieces, $V_2$ is the set of all simple (directed) cycles consisting of at least 3 edges and for $v\in V_1$ and $w\in V_2$ there is an edge $f\in F$ if $v\subseteq w$, i.e. the piece $v$ is a part of the simple cycle $w$. Notice that $(V,F)$ has exactly two connected components. The label $\lambda(v)$ of a piece $v$ is its size $|v|$. Notice that if there is an edge $f\in F$ between $p\in V_1$ and $c\in V_2$ there is also an edge (which may be denoted by $-f$) between $-p\in V_1$ and $-c\in V_2$, and $\lambda(-p)=\lambda(p)$.

Notice that $E=\bigcup_{p\in V_1} p$. For each $p\in V_1$ we enumerate the edges of $p$ as $\{e_1^p,\ldots,e_{|p|}^p\}$ in such a way that $-e_i^p=e_i^{-p}$, for every $p\in V_1$ and $i\leq |p|$. We define an equivalence relation $\sim$ on $V_1$, where $p\sim q$ if and only if $q=-p$. We denote the equivalence class of $p$ by $[p]$ and the set of equivalence classes by $[V_1]$. Similarly, the equivalence relation $\sim$ extends naturally to $V_2$ and $F$.

Set \[S:=\prod_{[p]\in [V_1]} S_{|p|},\] where for each $[p]\in [V_1]$, $S_{|p|}$ is the group of permutations of $\{1,\ldots,|p|\}$. 

Denote also by $([V],[F],\lambda)$ the bipartite graph with $[V]$ as the set of vertices and where there is an edge from $[F]$ between $[p_1],[p_2]\in [V]$ if and only if there is an edge from $F$ between $p_1$ and $p_2$ or between $p_1$ and $-p_2$. Notice that this is well-defined since the set of edges $F$ is also symmetric. Notice that $([V],[F])$ is connected. The piece labelling $\lambda$ of $V_1$ induces a piece labelling of $[V_1]$ since it is also symmetric.

Each automorphism $[\psi]$ of $([V],[F],\lambda)$ leaves $[V_1]$ invariant, thus acts on $[V_1]$  by permutations. By this we can consider a wreath product $S\wr\mathrm{Aut}([V],[F],\lambda)$ as in Definition~\ref{def:wreath}, where we note that for $[p],[q]\in [V_1]$ lying in the same orbit of $\mathrm{Aut}([V],[F],\lambda)$ we have $|p| = |q|$ and therefore we may in the definition of the wreath product take $\phi_{[p],[q]}:S_{|p|}\to S_{|q|}$ to be the identity.

\begin{proposition}
Given a $2$-connected undirected graph $M$, there is a topological group isomorphism $\Omega:\mathrm{LIso}(\F(M))\to\big(S\wr\mathrm{Aut}([V],[F],\lambda)\big)\times \{1,-1\}$.
\end{proposition}

\begin{proof}
Let $T:\F(M)\to\F(M)$ be a linear isometry. We shall define a triple $\Omega(T)=:((s_{[p]}^T)_{[p]\in [V_1]},[\psi_T],\varepsilon_T)$, where $[\psi_T]\in\mathrm{Aut}([V],[F],\lambda)$, $(s_{[p]}^T)_{[p]\in [V_1]}\in S$, and $\varepsilon_T\in\{1,-1\}$.

We first define a symmetric automorphism $\psi_T$ of $(V,F,\lambda)$; i.e. an automorphism satisfying $\psi_T(-x)=-\psi_T(x)$ for each $x\in V$. It will canonically induce an automorphism $[\psi_T]$ of $([V],[F],\lambda)$. By Proposition~\ref{prop:bijectionEext}, $T$ induces a symmetric simple cycle preserving bijection $\sigma_T:E\to E$, which in turn induces a Boolean algebra automorphism $\pi_T:B\to B$, which is itself symmetric. Since $\sigma_T$ is simple cycle preserving it induces a bijection $\psi_2:V_2\to V_2$. Since $\pi_T$ is a Boolean algebra automorphism it induces a bijection $\psi_1:V_1\to V_1$ and since $\pi_T$ is induced by $\sigma_T$, we have $\lambda(\psi_1(v)) = \lambda(v)$ for $v\in V_1$. We define $\psi_T: V\to V$ by $\psi_1\cup\psi_2$.

It is easy to observe that $(v,w)\in F$ if and only if $(\psi_T(v),\psi_T(w))\in F$.
Since $\psi_T$ is symmetric, $[\psi_T]\in\mathrm{Aut}([V],[F],\lambda)$ is well-defined. Next we define the sign $\varepsilon_T$. Since $(V,F)$ has two connected components, $\psi_T$ either preserves both of them, or switches them. In the former case, we set $\varepsilon_T=1$, in the latter we set $\varepsilon_T=-1$.

Now we define $(s_{[p]}^T)_{[p]\in [V_1]}$. Fix $p\in V_1$. Recall that $p=\{e_1^p,\ldots,e_{|p|}^p\}$. We also have $\psi^{-1}_T(p)=\{e_1^{\psi_T^{-1}(p)},\ldots,e_{|\psi_T^{-1}(p)|}^{\psi_T^{-1}(p)}\}$, and $|p|=|\psi^{-1}_T(p)|$. Since $\sigma^{-1}_T[p]=\psi^{-1}_T(p)$, for every $i\leq |p|$ there is $j\leq |p|$ such that $\sigma_T(e_i^{\psi_T^{-1}(p)})=e_j^p$. By the choice of the enumeration of the edges of $p$ and $-p$ and since $\sigma_T$ is symmetric, we have $\sigma_T(e_i^{\psi_T^{-1}(-p)}) = \sigma_T(-e_i^{\psi_T^{-1}(p)}) = -e_j^p = e_j^{-p}$. Thus we may safely set $s_{[p]}^T(i)=j$. It is clear that $s_{[p]}^T$ is a permutation of $\{1,\ldots,|p|\}$.

Let us show that $\Omega$ is injective. Let $T_1\neq T_2\in\mathrm{LIso}(\F(M)$. If $\pi_{T_1}\neq \pi_{T_2}$, then there exists an atom of $B$ where $\pi_{T_1}$ and $\pi_{T_2}$ differ. Thus by definition, there exists $p\in V_1$, where $\psi_{T_1}(p)\neq\psi_{T_2}(p)$. If $[\psi_{T_1}(p)]=[\psi_{T_2}(p)]$, then $\Omega(T_1)_3=-\Omega(T_2)_3$, where $\Omega(T_i)_3$ is the projection onto the third coordinate, i.e. the element of $\{1,-1\}$. If $[\psi_{T_1}(p)]\neq [\psi_{T_2}(p)]$, then $\Omega(T_1)_1\neq \Omega(T_2)_1$, where $\Omega(T_i)_1\in\mathrm{Aut}([V],[F],\lambda)$ is the projection onto the first coordinate.

So suppose that $\pi_{T_1}=\pi_{T_2}$. However, since $\sigma_{T_1}\neq\sigma_{T_2}$ there exists $e\in E$ such that $\sigma_{T_1}(e)\neq\sigma_{T_2}(e)$. We have that $e=e_i^p$, for some unique $p\in V_1$ and $i\leq |p|$. It follows that $(\Omega(T_1)_2)_{[\psi_{T_1}(p)]}(i)\neq (\Omega(T_2)_2)_{[\psi_{T_2}(p)]}(i)$, where $(\Omega(T_j)_2)_{[\psi_{T_i}(p)]}\in S_{|p|}$ is the corresponding projection.\bigskip

Let conversely $((s_{[p]})_{[p]\in [V_1]},[\psi],\varepsilon)\in (S\wr\mathrm{Aut}([V],[F],\lambda))\times\{1,-1\}$. We shall construct a symmetric simple cycle preserving bijection $\sigma:E\to E$. By Proposition~\ref{prop:canonicalIsometries}, there exists then a linear isometry $T:\F(M)\to\F(M)$ such that $\sigma = \sigma_T$. It will be then easy to check that $\Omega(T)=((s_{[p]})_{[p]\in [V_1]},[\psi],\varepsilon)$.

Pick $e\in E$. There is a unique piece $p\in V_1$ such that $e\in p$. So there is $i\leq |p|$ such that $e=e_i^p$. Let $q\in V_1$ be the unique piece such that $[q]=[\psi]([p])$ and $q$ and $p$ are in the same connected component of $(V,F)$.
We define  
\begin{equation}\label{eq:sigmaFromAut}
\sigma(e)=\sigma(e_i^p):=e_{s_{[q]}(i)}^{\varepsilon q}.
\end{equation}
Notice that $-e=e_i^{-p}$ by definition and so $[\psi]([-p])=[\psi]([p])$, however we now have that $-q$ is in the same connected component as $-p$. Thus \[\sigma(-e)=e_{s_{[q]}(i)}^{-\varepsilon q}=-\sigma(e),\] showing that $\sigma$ is symmetric.

It is clear that $\sigma$ is injective.  To check it is surjective, pick $e\in E$ which is again equal to $e_i^p$, for some $p\in V_1$ and $i\leq |p|$. Let $q$ be the unique piece such that $[q]=[\psi^{-1}]([p])$ and $q$ and $p$ lie in the same connected component. Then there is $j\leq |p|$ such that $s_{[p]}(j)=i$. It is straightforward to check that then $\sigma(e_j^{\varepsilon q})=e_i^p$.

We need to check that $\sigma$ is simple cycle preserving. Let $c\in V_2$ be a simple directed cycle. The cycle $c$ is a disjoint union of some pieces $p_1,\ldots,p_n\in V_1$, where $\{p_1,\ldots,p_n\}=\{p\in V_1\colon \exists f\in F\;(f\text{ connects }p\text{ and }c)\}$. Since $[\psi]$ is a graph automorphism we have that $\{[\psi]([p_i])\colon i\leq n\}=\{[p]\in [V]\colon \exists [f]\in [F]\; ([f]\text{ connects }[\psi]([p])\text{ and }[\psi]([c]))\}$. Let $c'\in V_2$ and $q_1,\ldots,q_n\in V_1$ be such that $[c']=[\psi]([c])$, $[q_i]=[\psi]([p_i])$, for $i\leq n$, and they lie in the same connected component as $c$ and $\{p_1,\ldots,p_n\}$ if $\varepsilon=1$, and in the other component of $c$ and $\{p_1,\ldots,p_n\}$ if $\varepsilon=-1$. Notice that $c'=\bigcup_{i\leq n} q_i$. By definition of $\sigma$, for each $i\leq n$ we have $\sigma[p_i]=q_i$. It follows that \[\sigma[c]=\bigcup_{i\leq n} \sigma[p_i]=\bigcup_{i\leq n} q_i=c',\] and since $c$ was arbitrary, we get that $\sigma$ preserves simple cycles. Thus, there exists $T\in \mathrm{LIso}(\F(M))$ such that $\sigma = \sigma_T$ and by the above we easily deduce that $[\psi_T] = [\psi]$. Pick any $p\in V_1$, then we have that $\varepsilon_T=1$ iff $\sigma(p)$ is in the same connected component as $p$, which is by the above equivalent to the fact that $\varepsilon=1$. Finally, it is now easy to check that $(s_{[p]}^T)_{[p]\in [V_1]} = (s_{[p]})_{[p]\in [V_1]}$.\medskip

We have proved that $\Omega$ is a surjection, thus it is bijective. It remains to prove it is a continuous group homomorphism with a continuous inverse.

Since the map $T\mapsto\sigma_T$ is a topological group isomorphism by Proposition~\ref{prop:SigmaLIso}, it suffices to check that the map $\sigma\in\Sigma\mapsto \big((s^{T_\sigma}_{[p]})_{[p]},[\psi_{T_\sigma}],\epsilon_{T_\sigma}\big)$ is a homeomorphic homomorphism, where $\Sigma$, as in Proposition~\ref{prop:SigmaLIso}, is the topological group of all symmetric bijections of $\edg[M]$ satisfying \ref{it:wA1} and \ref{it:wA2}.

Let us first check it is a group homomorphism. Fix $\sigma_1,\sigma_2\in\Sigma$. In the sequel, we shall write $\sigma$ for $\sigma_1\circ\sigma_2$, $\big((s_{[p]})_{[p]},\psi,\epsilon\big)$ for $\big((s^{T_{\sigma_1\circ\sigma_2}}_{[p]})_{[p]},[\psi_{T_{\sigma_1\circ\sigma_2}}],\epsilon_{T_{\sigma_1\circ\sigma_2}}\big)$, and $\big((s^i_{[p]})_{[p]},\psi_i,\epsilon_i\big)$ for $\big((s^{T_{\sigma_i}}_{[p]})_{[p]},[\psi_{T_{\sigma_i}}],\epsilon_{T_{\sigma_i}}\big)$, $i=1,2$. It is easy to deduce that $\psi_{T_{\sigma_1}\circ T_{\sigma_2}} = \psi_{T_{\sigma_1}}\circ \psi_{T_{\sigma_2}}$. Thus, we have $[\psi] = [\psi_1]\circ [\psi_2]$ and $\varepsilon=\varepsilon_1\varepsilon_2=:\delta$.

Since $\big((s^1_{[p]})_{[p]},\psi_1,\epsilon_1\big) \big((s^2_{[p]})_{[p]},\psi_2,\epsilon_2\big)=\big((s^1_{[p]}s^2_{[\psi_1^{-1}(p)]})_{[p]},\psi_1\circ\psi_2,\epsilon_1\varepsilon_2\big)$ we need to verify that $s_{[p]}=s^1_{[p]} s^2_{[\psi_1^{-1}(p)]}$ for every $[p]\in[V_1]$.

Fix therefore $[p]\in [V_1]$ and $i\leq |p|$. 
We have
\[\begin{split}
s^1_{[p]} s^2_{[\psi_1^{-1}(p)]}(i) = j & \Leftrightarrow e_j^p = \sigma_1\Big(e_{k}^{\psi_1^{-1}(p)}\Big) \; \& \; e_k^{\psi_1^{-1}(p)} = \sigma_2\Big(e_i^{\psi_2^{-1}\psi_1^{-1}(p)}\Big)\\
& \Leftrightarrow e_j^p = \sigma\Big(e_i^{\psi^{-1}(p)}\Big)\Leftrightarrow s_{[p]}(i) = j,
\end{split}\]
which proves that $s_{[p]}=s^1_{[p]} s^2_{[\psi_1^{-1}(p)]}$ for every $[p]\in[V_1]$ and therefore $\Omega$ is a group homomorphism.

It remains to show that the mapping $\Sigma\ni \sigma \mapsto \big((s^{T_\sigma}_{[p]})_{[p]},[\psi_{T_\sigma}],\epsilon_{T_\sigma}\big)$ is homeomorphism, which however easily follows from the fact that the topology on $\edg[\MM]$ is discrete by Lemma~\ref{lem:graphs} and so a net $(\sigma_\gamma)$ converges pointwise if and only if it is pointwise eventually constant.
\end{proof}

\begin{remark}
Although this section was formulated, in order to ease the notation, only for Lipschitz-free spaces over connected graphs, all the results are valid as well for $\Lip_0$-spaces over connected graphs. In particular, we have a complete description of the linear isometry group of $\Lip_0(G)$, whenever $G$ is a connected graph. This follows, by Remark~\ref{rem:rigidGraphs}, from the fact that $\Lip_0(G)$ has a strongly unique predual.
\end{remark}

\section{Examples}\label{sec:examples}
This penultimate section complements the theory developed in this paper. In the first part of the section we show that a certain interesting and extensively studied class of metric spaces in metric geometry belongs to the class of Lipschitz-free rigid Prague spaces. In the second part we present examples showing that converses to various results from Section~\ref{sec:rigid} do not hold.
 
\subsection{Carnot groups as examples of Lipschitz-free rigid Prague spaces}
The theorem below presents one of our main examples of Lipschitz-free rigid spaces. We note that the rigidity of the same class for the $1$-Wasserstein spaces was obtained in \cite{BTV-Carnot} - by different methods.
\begin{theorem}\label{thm:carnotExample}
Let $(\bG,d_N)$ be a non-abelian Carnot group endowed with a metric induced by a homogeneous horizontally strictly convex norm. Then $(\bG,d_N)$ is a Prague space, which is moreover Lipschitz-free rigid.
\end{theorem}

The class of spaces considered in Theorem~\ref{thm:carnotExample} includes many important metric spaces considered by numerous authors, see Examples~\ref{ex:carnotExamples} where more details are provided.

Before we come to the proof of Theorem~\ref{thm:carnotExample}, we recall the  relevant information about Carnot groups. Every Carnot group is as a topological group homeomorphic to $\Rea^n$. For the informed reader, we note that more precisely every Carnot group is a simply-connected and connected nilpotent Lie group whose Lie algebra admits a stratification. We provide more comments on this in Remark~\ref{rem:carnotRef}. The properties of Carnot groups we need are the following.

\begin{enumerate}[label=(Car\arabic*)]
    \item\label{it:carOne} Carnot group as a topological group is (canonically) isomorphic to $\bG = (\Rea^N,*)$, where on $\Rea^N$ we have the usual Euclidean topology and $*$ is a real analytic group operation defined by a polynomial. After this (canonical) identification, we can decompose the Carnot group $(\Rea^N,*)$ as $\Rea^N = \Rea^{N_1}\oplus\ldots\oplus\Rea^{N_r}$ with $N_1+\dots+N_r=N$, where $r>1$ if and only if $\bG$ is not abelian, and find a family of Carnot group isomorphisms $\{\delta_{\lambda}\}_{\lambda>0}$ (called Carnot-dilations) such that: $$\delta_{\lambda}(x^{(1)},\dots,x^{(r)})=(\lambda x^{(1)},\dots,\lambda^r x^{(r)}),$$
where $x^{(i)}\in\mathbb{R}^{N_i}$ for $i=1,\dots,r$.
    \item\label{it:HomogeneousMetric}  A map $N: \bG \to\mathbb{R}_{\geq0}$ is called a norm on $\bG$ if it satisfies
\begin{itemize}
\item[(i)] $N(g)=0\quad\Longleftrightarrow\quad g=(0,\dots,0)$,
\item[(ii)] $N(g^{-1})=N(g)$, for all $g\in \bG$,
\item[(iii)] $N(g\ast g')\leq N(g)+N(g')$, for all $g,g'\in \bG$.
\end{itemize}
Every norm $N$ induces a left-invariant metric $d_{N}$ (i.e. metric satisfying $d_N(g_0\ast g,g_0\ast g')=d_N(g,g')$ for all $g_0,g,g'\in \bG$) by the formula $d_N(g,g') := N(g^{-1}\ast g')$. A norm $N:\bG \to\mathbb{R}_{\geq0}$ on a Carnot group is called \emph{homogeneous} if
$N(\delta_{\lambda}(g))=\lambda N(g)$ for all $\lambda>0$ and for all $g\in\mathbf{G}$.

Whenever $d_N$ is a metric induced by a homogeneous norm on $\bG = (\Rea^N,*)$, then $d_N$ is compatible with the topology of $\bG$ and $(\bG,d_N)$ is proper and complete.

\item\label{it:convexNorm} Let $\bG = (\Rea^N,*)$ be a Carnot group as in \ref{it:carOne} (so on $\bG$ we have the decomposition and Carnot-dilations as in \ref{it:carOne}). Given $z\in \Rea^{N_1}$, we say $l(z):=\{(sz,0,\dots,0)\setsep s\in\Rea\}\subset \bG$ is \emph{horizontal line through the origin}. It is a subgroup of $\bG$ (in fact, it is what is called a \emph{one-parameter subgroup}). In particular, for $s,t\in\Rea$, we have $(sz,0,\ldots,0)*(tz,0,\ldots,0)=((s+t)z,0,\ldots,0)$. The set of all the points on all the horizontal lines through the origin is denoted as $H_\bG$, that is, $H_\bG:=\bigcup_{z\in \Rea^{N_1}}\{l(z)\}=\{(z,0,\ldots,0)\colon z\in \Rea^{N_1}\}$. Notice therefore that $\bG\neq H_\bG$ if and only if $\bG$ is not abelian.

A homogeneous norm $N$ on $\bG=(\Rea^N,*)$ is said to be \emph{horizontally strictly convex} if whenever $g,g' \in \bG\setminus\{(0,...,0)\}$ are such that $N(g\ast g') = N(g) +N(g')$, then there exists $l(z)\in H_\bG$ such that $g,g'\in l(z)$.

Suppose that $d_N$ is the metric induced by a homogeneous horizontally strictly convex norm on $\bG=(\Rea^N,*)$. Then it is easy to check that whenever $g_1,g_2,g \in \bG$ ($g\neq g_1$, $g\neq g_2$) are such that $d_N(g_1, g_2) = d_N (g_1, g) + d_N (g, g_2)$, there exists $z\in\Rea^{N_1}$ such that $g_1^{-1}*g_2\in l(z)$, in particular $g_2\in g_1*H_\bG:=\{g_1*h\colon h\in H_\bG\}$.

We note that on any Carnot group there exists a homogeneous horizontally strictly convex norm.
\end{enumerate}

\begin{remark}\label{rem:carnotRef}
In this remark we comment on Carnot groups for readers who have some knowledge of Lie theory. Carnot groups are simply-connected and connected nilpotent Lie groups $G$ whose Lie algebra $\mathfrak{g}$ admits a stratification, i.e. a direct sum decomposition $\mathfrak{g}=\bigoplus_{i=1}^r V_i$ such that for every $i<r$, $[V_1,V_i]=V_{i+1}$. Since $G$ and $\mathfrak{g}$ are nilpotent, the Baker-Campbell-Hausdorff formula has only finitely many non-vanishing terms and therefore may be used to globally define, using polynomials, group operation $\ast$ on $\mathfrak{g}$. Since $(\mathfrak{g},\ast)$ and $G$ are then two simply-connected connected Lie groups with Lie algebra $\mathfrak{g}$, they must be equal. This explains much of \ref{it:carOne}, i.e. why without loss of generality $G$ is $\Rea^N$ with polynomially defined multiplication and the decomposition $G=\Rea^N=\Rea^{N_1}\oplus\ldots\oplus\Rea^{N_r}$ corresponds to the stratification $V_1\oplus\ldots\oplus V_r$ on $\mathfrak{g}$, where the dimension of $V_i$, for $i\leq r$, is $N_i$. It also explains why for a fixed $z\in\Rea^{N_1}$, $l(z)$ is a subgroup. Indeed, given $s,t\in\Rea$, $sz$ and $tz$ commute as elements of $\mathfrak{g}$, i.e. $[sz,tz]=0$, so by the Baker-Campbell-Hausdorff formula we have $((s+t)z,0\ldots,0) = (sz,0\ldots,0)*(tz,0\ldots,0)$. Note also that by definition, $r=1$ is equivalent to $[\mathfrak{g},\mathfrak{g}]=\{0\}$, which in turn is equivalent to the fact that $\bG$ is abelian.

We refer the reader to \cite{HilNeeb} for a general background on Lie groups, including nilpotency and the Baker-Campbell-Hausdorff formula.  We also refer to \cite{CarnotBook} which is more specialized to Carnot groups.
In the paragraph below we suggest some references, where the interested reader may find proofs of the above mentioned properties of Carnot groups together with some more interesting results.\medskip

The property \ref{it:HomogeneousMetric} follows from \cite[Proposition 2.26]{DS19}. The notion of horizontally strictly convex homogeneous norm mentioned in \ref{it:convexNorm} was introduced in \cite{BFS18} and it was proved in \cite{BTV-Carnot} that such a norm exists on every Carnot group. For some natural examples we refer to Examples~\ref{ex:carnotExamples}.
\end{remark}

\begin{proof}[Proof of Theorem~\ref{thm:carnotExample}]Let $\bG = (\bG,d_N)$ be as in the assumptions. We may without loss of generality assume $\bG=(\Rea^N,*)$ as in \ref{it:carOne}. First, we claim that
$$\edg[\bG]=\{(g,g')\colon g'\notin g*H_\bG\}.$$
Indeed, fix two distinct points $g,g'\in\bG$. Since $\bG$ is proper, using \cite[Proposition 2.3]{AP20} we obtain that  $m_{g,g'}\in \F(\bG)$ is a preserved extreme point if and only if $[g,g'] = \{g,g'\}$ which is in turn equivalent to the fact that  $g'\notin g* H_\bG$.
 
Next we show that for an arbitrary element $g\in\bG$ and for arbitrary $\varepsilon>0$ we can choose a $g_{\varepsilon}\in\bG$ such that $(g,g_{\varepsilon})\in\edg[\bG]$ and $d_{N}(g,g_{\varepsilon})<\varepsilon$. Put $H_{\bG}(g):=g\ast H_{\bG}$. Since 
left-translation is an isometry (and thus a homeomorphism), we get that $H_{\bG}(g)$ is a closed set with empty interior. Here we used that $d_N$ is compatible with the Euclidean topology of $\mathbb{R}^N$, and that $H_{\bG}$ has an empty interior since $\mathrm{dim}(H_{\bG})=N_1<N$ as $\bG$ is not abelian. On the other hand, $B(g,\varepsilon)$ (the $d_{N}$-ball centered at $g\in\bG$ with radius $\varepsilon$) is an open set, and therefore $B(g,\varepsilon)\setminus H_{\bG}(g)\neq\emptyset$, and for all $g_{\varepsilon}\in B(g,\varepsilon)\setminus H_{\bG}(g)$, using the observation in the paragraph above, we have $(g,g_{\varepsilon})\in\edg[\bG]$. Consequently, $\ver[\bG]=\bG$.\\

Finally, if $g_1,g_2\in \bG$, $g_1\neq g_2$, then for all $\varepsilon>0$ the same argument as above gives us an element $g_{\varepsilon}\in B(g_1,\varepsilon)\setminus\big(H_{\bG}(g_1)\cup H_{\bG}(g_2)\big)$. For such a $\geps$  we have
$d_N(\geps,g_2)\leq d_N(\geps,g_1) + d_N(g_1,g_2)<\varepsilon+d_N(g_1,g_2)$ and thus $d_N(g_1,\geps)+d_N(\geps,g_2)\in [d_N(g_1,g_2), 2\varepsilon+d_N(g_1,g_2)]$.
Using this simple observation and the fact that $(g_1,\geps),(g_2,\geps)\in\edg[\bG]$ we have
\begin{equation*}
\begin{split}
d_N(g_1,g_2)&\leq \inf\Big\{\sum_{i=1}^k d_N(e_i)\colon (e_i)_{i=1}^k\in \edg[\bG]^k\text{ is a walk from }g_1\text{ to }g_2\Big\}\\
&\leq\inf_{\varepsilon}\{d_N(g_1,\geps)+d_N(\geps,g_2)\}=d_N(g_1,g_2).
\end{split}
\end{equation*}

This shows that $\bG$ is a Prague space. Moreover, since by the above there are infinitely many paths between any two points, we have that in particular $\edg[\bG]$ is $3$-connected and so $\bG$ is Lipschitz-free rigid by Corollary~\ref{cor:3ConnectedGraph}.
\end{proof}

\begin{examples}\label{ex:carnotExamples}Let us mention important examples of metrics on Carnot groups to which our Theorem~\ref{thm:carnotExample} applies.
\begin{itemize}
    \item Recall that the Heisenberg group $\HH$ is the Carnot group $\HH=(\Rea^{2n+1},*)$, where for $x,y\in\Rea^n$ and $t\in\Rea$ we have
$$(x,y,t)\ast (x',y', t') := (x+x',y+y',t+t'+2\sum_{i=1}^n (x'_iy_i-x_iy'_i))$$
and Carnot-dilations $(\delta_\lambda)_{\lambda > 0}$ are given by $\delta_\lambda(x,y,t):=(\lambda x,\lambda y,\lambda^2 t)$. Examples of horizontally strictly convex homogeneous norms on $\HH$ are e.g. the Heisenberg-Kor\'anyi norm $\norm{\cdot}_H$ given by $\|(x,y,t)\|_H :=
\big(\norm{(x,y)}_E^4 + t^2\big)^{\frac{1}{4}}$ for $(x,y,t)\in\HH$ or the Lee and Naor norm $\norm{\cdot
}_{LN}$ given by $\|(x,y,t)\|_{LN}=\sqrt{\|(x,y,t)\|_H^2 + \|(x,y)\|_E^2}$ for $(x,y,t)\in\HH$, where $\norm{\cdot}_E$ denotes the Euclidean norm. We refer to \cite{BFS18}  for more details and examples of horizontally strictly convex homogeneous norms on $\HH$.
\item Given a Carnot group $\bG = (\Rea^N,*)$, the Hebisch-Sikora norm on $\norm{\cdot}_{HS}:\bG \to\Rea_+$ is defined as
\[
\norm{q}_{HS}:=\inf\{t > 0\colon \delta_{1/t}(q)\in B(0, r)\},\qquad q\in \bG,
\]
where $B(0, r)\subset \Rea^N$ is the usual Euclidean ball in $\Rea^N$ centered at the origin, with radius $r$. In \cite{BTV-Carnot} the authors proved that there exists $r_0 > 0$ such that for all $0 < r < r_0$ the function $\norm{\cdot}_{HS}$ defined above is a horizontally stricly convex homogeneous norm on $\bG$. We refer to \cite{BTV-Carnot} and references therein for more information about Hebish-Sikora norms and their applications.
\end{itemize}
Finally, we note that our Theorem~\ref{thm:carnotExample} does not apply to Carnot groups endowed with the standard Carnot-Carath\'eodory distance, because geodesic metric spaces are not Prague spaces. However, we note that the Carnot-Carath\'eodory distance on Heisenberg groups is the path metric associated to the Kor\'anyi norm (see \cite[Corollary 2.3.5]{Lu-thesis}). We do not know if there is such a close relation between the Carnot-Carath\'eodory distance and other horizontally strictly convex norms.
\end{examples}

\subsection{Counter-examples from Section~\ref{sec:rigid}}

In the first proposition we prove the claim mentioned in Example~\ref{ex:rigidNot2Connected}.
\begin{proposition}\label{prop:rigidNot2Connected}
Consider the subset $I:=\{0\}\times [0,1] \subset \Rea^2$ and point $x_0:=(1,0)\in \Rea^2$. Then the metric space $\MM = (I\cup \{x_0\},\|\cdot\|_2)$ is a weak Prague space which is Lipschitz-free rigid, but $\edg[\MM]$ is not $2$-connected.
\end{proposition}
\begin{proof}
For any $x,y\in \MM$ we have $[x,y] = \{x,y\}$ if and only if $x_0\in \{x,y\}$. Fixing $y\in I$ and picking any $z\neq y\in I$ we have $d(x_0,y)+d(x_0,z)\geq 2$ so clearly by Fact~\ref{fact:extremePoints}, $m_{x_0,y}\in \ext B_{\F(\MM)^{**}}$. Therefore, we have $\edg = \{(x,x_0),(x_0,x)\colon x\in I\}$ and $\ver = \MM$. Thus, we easily observe that $\edg$ is weakly admissible and since there are no cycles in $\edg$, $\edg$ is not $2$-connected. It remains to check $\MM$ is Lipschitz-free rigid.

Let $\sigma:\edg\to\edg$ be a bijection satisfying conditions \ref{it:wA1}, \ref{it:wA2} and \ref{it:wA3}. By Theorem~\ref{thm:1-1correspondencePrague}, it suffices to check that then there exists $\varepsilon\in \{\pm 1\}$ such that $\sigma(e)=\varepsilon e$, $e\in \edg$. Since $\sigma$ is a bijection, there exists a bijection $f:[0,1]\to [0,1]$ and $\varepsilon:[0,1]\to \{\pm 1\}$ such that $\sigma(x_0,(0,t)) = \varepsilon(t)(x_0,(0,f(t)))$ for every $t\in [0,1]$.

First, we shall show that $\varepsilon$ is a constant mapping. For $t>0$, using \ref{it:wA3} we obtain that for the $\edg$-path $((0,0),x_0),(x_0,(0,t))$ from $(0,0)$ to $(0,t)$ we have 
\[\begin{split}
t & = \normb{Big}{\varepsilon(0)d(x_0,(0,0))m_{(0,f(0)),x_0} + \varepsilon(t)d(x_0,(0,t))m_{x_0,(0,f(t))}} \\
& = \normb{Big}{\varepsilon(0)\frac{d(x_0,(0,0))}{d(x_0,(0,f(0)))}\delta_{(0,f(0))} - \varepsilon(t)\frac{d(x_0,(0,t))}{d(x_0,(0,f(t)))}\delta_{(0,f(t))}},
\end{split}\]
which implies that $\varepsilon(t) = \varepsilon(0)$ as otherwise we would have \[t = \normb{Big}{\frac{d(x_0,(0,0))}{d(x_0,(0,f(0)))}\delta_{(0,f(0))} + \frac{d(x_0,(0,t))}{d(x_0,(0,f(t)))}\delta_{(0,f(t))}} = d(x_0,(0,0)) + d(x_0,(0,t)) > 1,\]
which is in contradiction with $t\leq 1$. Since $t>0$ was arbitrary, this proves that $\varepsilon$ is a constant map, so we can without loss of generality assume that $\sigma(x_0,(0,t)) = (x_0,(0,f(t)))$ for every $t\in [0,1]$ and it remains to prove that $f$ is the identity.

Similarly as above, given $s,t\in[0,1]$ and $\edg$-path $((0,s),x_0),(x_0,(0,t))$ from $(0,s)$ to $(0,t)$, using \ref{it:wA3} we obtain
\[
|t-s| = \normb{Big}{\frac{d(x_0,(0,s))}{d(x_0,(0,f(s)))}\delta_{(0,f(s))} - \frac{d(x_0,(0,t))}{d(x_0,(0,f(t)))}\delta_{(0,f(t))}}.\]
By the well-known formula \cite[Lemma 11]{CuthJohanis}, supposing that $\frac{1+t^2}{1+f(t)^2}\geq \frac{1+s^2}{1+f(s)^2}$ we have
\begin{equation}\label{eq:exNot2Connected}
|t-s| = \sqrt{1+t^2} + \Big(|f(t)-f(s)| - \sqrt{1+f(t)^2}\Big)\sqrt{\frac{1+s^2}{1+f(s)^2}}
\end{equation}
(and of course if $\frac{1+t^2}{1+f(t)^2}\leq \frac{1+s^2}{1+f(s)^2}$ then the above holds with $s$ and $t$ interchanged). We shall prove in a series of claims that this already implies that $f$ is the identity mapping.
\begin{claim}
The function $f$ is $\sqrt{2}$-Lipschitz. Moreover, it is a strictly increasing homeomorphism with $f(0)=0$ and $f(1)=1$.
\end{claim}
\begin{proof}[Proof of Claim 1]
Supposing that $\frac{1+t^2}{1+f(t)^2}\geq \frac{1+s^2}{1+f(s)^2}$, using \eqref{eq:exNot2Connected} and simplifying we obtain
\[\begin{split}
|t-s|(\sqrt{1+f(s)^2}) + \sqrt{(1+s^2)(1+f(t)^2)} & = \sqrt{(1+t^2)(1+f(s)^2)} + |f(t)-f(s)|\sqrt{1+s^2}\\
& \geq \sqrt{(1+s^2)(1+f(t)^2)} + |f(t)-f(s)|\sqrt{1+s^2},
\end{split}\]
which implies that
\[
\frac{|f(t)-f(s)|}{|s-t|}\leq \sqrt{\frac{1+f(s)^2}{1+s^2}}
\]
and, changing the roles of $t$ and $s$ if $\frac{1+t^2}{1+f(t)^2}\geq \frac{1+s^2}{1+f(s)^2}$ does not hold, we obtain that for any $t,s\in [0,1]$ we have
\[
\frac{|f(t)-f(s)|}{|s-t|}\leq \max\Big\{\sqrt{\frac{1+f(t)^2}{1+t^2}}\colon t\in [0,1]\Big\} \leq \sqrt{\frac{1+1}{1+0}} = \sqrt{2}.
\]
Thus, the function $f$ is $\sqrt{2}$-Lipschitz and therefore $f:[0,1]\to [0,1]$ is continuous bijection, so it is a homeomorphism and so it is either srictly increasing function with $f(0)=0$ and $f(1)=1$ or strictly decreasing function with $f(1)=0$ and $f(0)=1$. Since for $t=1$ and $s=0$ we have 
\[
\frac{1+t^2}{1+f(t)^2} = \frac{2}{1+f(1)^2} \geq 1\geq \frac{1}{1+f(0)^2}= \frac{1+s^2}{1+f(s)^2},
\]
from \eqref{eq:exNot2Connected} we obtain 
\[
1=\sqrt{2} + (1-\sqrt{1+f(1)^2})\frac{1}{\sqrt{1+f(0)^2}}
\]
and since we have $\{f(0),f(1)\} = \{0,1\}$, this implies that $f(0)=0$ and $f(1)=1$ and, by the above, the function $f$ is strictly increasing.
\end{proof}

\begin{claim}
We have $f(t) = t$ for every $t\in [0,1]$.
\end{claim}
\begin{proof}[Proof of Claim 2]
Assume that we have $t\geq f(t)$. Then since $\tfrac{1+t^2}{1+f(t)^2}\geq 1 = \tfrac{1+0^2}{1+f(0)^2}$, we obtain from \eqref{eq:exNot2Connected} applied to $t$ and $s=0$ that, using Claim 1 we have
\[
t = \sqrt{1+t^2} + f(t) - \sqrt{1+f(t)^2},
\]
which implies that for the function $h(x):=x-\sqrt{1+x^2}$ we have $h(t) = h(f(t))$, but since $h$ is one-to-one on $[0,1]$, it follows that $t=f(t)$.

Finally, assume that $t\leq f(t)$. Then since $1 = \tfrac{1+0^2}{1+f(0)^2}\geq \tfrac{1+t^2}{1+f(t)^2}$, from \eqref{eq:exNot2Connected} we obtain
\[
t = 1 + (f(t) - 1)\sqrt{\frac{1+t^2}{1+f(t)^2}},
\]
which implies that for the function $g(x):=\tfrac{x-1}{\sqrt{1+x^2}}$ we have $g(t) = g(f(t))$, but since $g$ is one-to-one on $[0,1]$, it follows that $t=f(t)$.
\end{proof}

Thus, $f$ is identity and so $\MM$ is Lipschitz-free rigid.
\end{proof}

Next we provide the example mentioned in Remark~\ref{rem:rigidNot3Connected}.
\begin{proposition}\label{prop:rigidNot3Connected}
There exists a finite graph $G=(V,E)$ which is $2$-connected, but not $3$-connected such that for every simple cycle-preserving bijection $\sigma:E\to E$ there exists a graph isomorphism $f:V\to V$ such that $\sigma(\{x,y\}) = \{f(x),f(y)\}$ for every $\{x,y\}\in E$.

In particular, $G$ is a Lipschitz-free rigid space with $\edg$ not being $3$-connected.
\end{proposition}
\begin{proof}
In this proof by $K_n$ we denote complete graph with $n$ vertices. Our graph $G$ consists of three disjoint complete graphs $K_{i_1}$, $K_{i_2}$ and $K_{i_3}$ together with three more vertices $a_1$, $a_2$ and $a_3$, where each $a_j$, $j\in \{1,2,3\}$ is connected by edges $E_{j,k}$ with vertices from $K_{i_k}$, $k\in \{1,2,3\}$, where $E_{j,j} = \emptyset$ and $|E_{j,k}|\geq 3$ are distinct natural numbers.

\begin{center}
\includegraphics[scale=0.8]{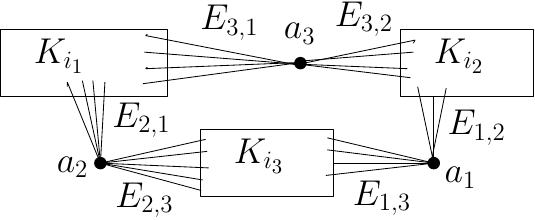}
\end{center}
We claim that if moreover $\min\{i_1,i_2,i_3\} > 2 + \max\{|E_{(j,k)}|\colon j,k\in \{1,2,3\}, j\neq k\}$, then this will be the graph we are looking for.

It is easy to observe that $G$ is $2$-connected and not $3$-connected. Moreover, if we denote by $G_j$, $j=1,2,3$ subgraphs with edges $\bigcup_{k=1}^3 E_{k,j}\cup K_{i_j}$, then each $G_j$ is $3$-connected. Let $\sigma:E\to E$ be a cycle-preserving bijection. We claim that then $\sigma(G_j)=G_j$ for every $j=1,2,3$.

Indeed, first we observe that $e\in E_{j,k}$ if and only if $e$ is contained in exactly $|E_{j,k}|-1$ cycles of length three, which is a property preserved by $\sigma$, and since $|E_{j,k}|$ are distinct numbers, we have 
\begin{equation}\label{eq:ejkFixed}
\sigma(E_{j,k}) = E_{j,k},\qquad j,k=1,2,3.
\end{equation}
Next, for each distinct edges $e,e'\in E_{j,k}$ there is a unique $f\in K_{i_k}$ such that $\{e,e',f\}$ is a cycle and since by \eqref{eq:ejkFixed} we have $\{\sigma(e),\sigma(e')\}\subset E_{j,k}$, we obtain $\sigma(f)\in K_{i_k}$ because $\{\sigma(e),\sigma(e'),\sigma(f)\}$ is a cycle. Thus,
\begin{equation}
\forall k\in\{1,2,3\} \; \exists f\in K_{i_k}: \sigma(f)\in K_{i_k}.
\end{equation}
Now, pick any $f'\in K_{i_k}$ and consider now the edge $f\in K_{i_k}$ with $\sigma(f)\in K_{i_k}$. Then there are at least $i_k-2$ disjoint cycles containing both $\{f,f'\}$, so since we have $\sigma(f)\in K_{i_k}$ and $\max\{|E_{j,k}|\colon j=1,2,3\}<i_k-2$, we obtain that $\sigma(f')\in K_{i_k}$. Since $f'\in K_{i_k}$ was arbitrary, we have $\sigma(K_{i_k}) = K_{i_k}$ for $k=1,2,3$ which together with \eqref{eq:ejkFixed} gives that $\sigma(G_j)=G_j$ for $j=1,2,3$, which finishes the proof of the claim above.

Since each $G_j$ is $3$-connected, by Corollary~\ref{cor:3ConnectedGraph} we obtain that there are graph isomorphisms $f_j:G_j\to G_j$ such that $\sigma(\{x,y\}) = \{f_j(x),f_j(y)\}$ for each edge $\{x,y\}$ from $G_j$. Since $|E_{j,k}|$ are distinct numbers, vertices $a_k$ in graphs $G_j$ are uniquely determined by their degrees, and so $f_j|_{\{a_1,a_2,a_3\}\setminus\{a_j\}}$ is identity. Thus, $f = f_1\cup f_2\cup f_3:G\to G$ is a well-defined graph isomorphism satisfying $\sigma(\{x,y\}) = \{f(x),f(y)\}$ for any edge $\{x,y\}$ from $G$. This finishes the proof that $G$ is the graph we wanted to find.

The ``In particular" part then immediately follows. Indeed, using the same arguments as in the proof of Theorem~\ref{thm:3ConnectedImpliesRigid} (together with Lemma~\ref{lem:Eext=E}) we see that $G$ is a Lipschitz-free rigid Prague space.
\end{proof}

\section{Concluding remarks and open problems}\label{sec:concluding}
\subsection{Remarks}
A famous open problem asks whether $\FF(\Rea^n)$ and $\FF(\Rea^m)$, and also $\F(\Int^n)$ and $\F(\Int^m)$, are isomorphic if $n,m\geq2$, $n\neq m$. We cannot answer this exact question but our results provide answers to some other related open problems.\medskip

\noindent{\bf (1)} Consider $\Int^n$, for $n\geq 2$, equipped either with the inherited euclidean distance or the $\ell_1$-distance, also known as the Manhattan distance in this context, which is the graph metric for the canonical graph structure on $\Int^n$. It is routine to check that it is a Prague space with $\edg$ being $3$-connected - in the case of the Manhattan distance, simply a $3$-connected graph. By Theorem~\ref{thm:3ConnectedImpliesRigid}, $\F(\Int^n)$ is Lipschitz-free rigid and since $\Int^n$ is uniformly discrete and so does not admit non-isometric dilations, it follows that $\LIso(\F(\Int^n))$ is equal to the product of $\{-1,1\}$ with the isometry group $\mathrm{Iso}(\Int^n)$ of $\Int^n$ (see Proposition~\ref{prop:rigidLIso}), which is the same for both the Euclidean and Manhattan distances. The group $\mathrm{Iso}(\Int^n)$ as a semidirect product $F_n\ltimes T_n$, where $T_n$ is the group of translations equal to $\Int^n$ itself and $F_n$ is the finite group of isometries fixing $0$ which consists of rotations and reflections. In particular, $\LIso(\F(\Int^n))$ and $\LIso(\F(\Int^m))$ differ for different $n\neq m\geq 2$. Indeed, assuming $n>m$, $\LIso(\F(\Int^m))$ is equal to $\{-1,1\}\times (F_m\ltimes \Int^m)$ and does not contain $\Int^n$ as a subgroup, while $\LIso(\F(\Int^n))$ does.

\begin{remark}
We were informed by Ram\'on Aliaga that it is possible to distinguish $\F(\Int^n)$ and $\F(\Int^m)$ isometrically also directly by counting the preserved extreme points. Such a direct approach seems no longer possible in the next more involved example where we really need to compute the isometry groups as an invariant that distinguishes the corresponding Lipschitz-free spaces.
\end{remark}

\noindent{\bf (2)} Consider for each $n\in\Nat$ the Heisenberg group $\He^n$ equipped with the Kor\' anyi metric, which is a topological group homeomorphic to $\Rea^{2n+1}$ (see Examples~\ref{ex:carnotExamples} for details). By Theorem~\ref{thm:carnotExample}, $\He^n$ is Lipschitz-free rigid, i.e. every isometry of $\F(\He^n)$ is induced by $\varepsilon\in\{-1,1\}$ and an $a$-dilation $\phi:\He^n\to\He^n$. In fact, according to Proposition \ref{prop:rigidLIso}, $\{-1,1\}\times \operatorname{Dil}(\mathbb{H}^n)\ni (\varepsilon,\phi)\mapsto \varepsilon T_{\phi}\in \LIso(\F(\mathbb{H}^n))$ is an isomorphism of topological groups. Notice that in the Heisenberg case an $a$-dilation can be uniquely written as the composition of an isometry and a Carnot-dilation. Indeed, using the homogeneity of the metric, we have that $d(x,y)=\frac{1}{a} d(\phi(x),\phi(y))=d(\delta_{\frac{1}{a}}\phi(x),\delta_{\frac{1}{a}}\phi(y))$, and thus $\phi(x)=\delta_{a}\Big(\delta_{\frac{1}{a}}\phi(x)\Big)$, where $\delta_{\frac{1}{a}}\phi(x)$ is an isometry. This means that $\LIso(\F(\mathbb{H}^n))=\{-1,1\}\times \operatorname{Dil}(\mathbb{H}^n)$ can be topologically identified with $\{-1,1\}\times \mathbb{R}_+\times\operatorname{Iso}(\mathbb{H}^n)$. We claim that if $m < n$ then the spaces $\F(\mathbb{H}^n)$ and $\F(\mathbb{H}^m)$ are not isometric. Indeed, if they were then their isometry groups would coincide. Our task is now to show that $\{-1,1\}\times \mathbb{R}_+\times\operatorname{Iso}(\mathbb{H}^n)$ and $\{-1,1\}\times \mathbb{R}_+\times\operatorname{Iso}(\mathbb{H}^m)$ are not isomorphic. By \cite[Theorem 3.5]{TanYang}, $\operatorname{Iso}(\mathbb{H}^n)$ is the same group regardless of $\He^n$ being equipped with the left-invariant Riemannian distance or the Carnot-Carath\' eodory metric. Furthermore, by \cite[Corollary 2.3.5]{Lu-thesis} it is also the same group if $\He^n$ is equipped with the Kor\' anyi metric. In particular, it follows from the Myers-Steenrod theorem (\cite{MySt}) that $\operatorname{Iso}(\mathbb{H}^n)$ is a Lie group. Since $\Rea_+$ and $\{-1,1\}$ are Lie as well and a product of Lie groups is a Lie group we get that $\LIso(\F(\He^n))$ is a Lie group. Therefore if $\LIso(\F(\He^m))$ and $\LIso(\F(\He^n))$ were isomorphic their Lie algebras $\mathfrak{l}_m$, resp. $\mathfrak{l}_n$ would be isomorphic (see e.g. \cite[Corollary 9.1.10]{HilNeeb}). We shall reach a contradiction by proving that the dimensions of $\mathfrak{l}_m$ and $\mathfrak{l}_n$ are not equal. By again \cite[Theorem 3.5]{TanYang} $\operatorname{Iso}(\mathbb{H}^n)$ is a semidirect product $A_n\ltimes \He^n$, where the right side is the Heisenberg group itself acting by translations and $A_n$ is the group of isometries of $\He^n$ fixing $0$. Denoting, for $i=m,n$, the Lie algebra of $A_i$ and $\He^i$ $\mathfrak{a}_i$, resp. $\mathfrak{h}_i$ we get that $\rm{dim}(\mathfrak{l}_i)=\rm{dim}(\mathfrak{a}_i)+\rm{dim}(\mathfrak{h}_i)+\rm{dim}(\Rea_+)=\rm{dim}(\mathfrak{a}_i)+\rm{dim}(\mathfrak{h}_i)+1$ (by e.g. \cite[Proposition 9.2.25]{HilNeeb}). We have $\rm{dim}(\mathfrak{h}_i)=2i+1$, so we need to compute $\rm{dim}(\mathfrak{a}_i)$. Since by \cite[Corollary 3.6 (1)]{TanYang} the connected component of the unit of $A_i$ can be identified with the unitary group $U(i)$, we get that $\mathfrak{a}_i$ is isomorphic to the Lie algebra of $U(i)$ which is known to have dimension $i^2$ (see e.g. \cite[Example 5.1.6 (v)]{HilNeeb}). Therefore we get that $\rm{dim}(\mathfrak{l}_i)=i^2+2i+2$ and we are done.

\subsection{Problems}
We find the following question the most pressing.
\begin{question}
Does there exist a Lipschitz-free rigid metric space $\MM$ such that $\bigcup\edg[\MM]$ is not dense in $\MM$?
\end{question}

A very interesting answer to the previous question would be answering positively the next question. However, even a negative answer to the next question would be interesting.

\begin{question}
Is $\Rea^d$, for $d\geq 2$, Lipschitz-free rigid? Either if it is equipped with the euclidean or with the Manhattan distance.
\end{question}

\begin{question}
Does every metric space isometrically embed into a Lipschitz-free rigid space that contains only one additional point?
\end{question}

\bibliography{ref}\bibliographystyle{siam}

\begin{thebibliography}{10}

\bibitem{AFGZ}
{\sc M.~Alexander, M.~Fradelizi, L.~C. Garc\'{\i}a-Lirola, and A.~Zvavitch},
  {\em Geometry and volume product of finite dimensional {L}ipschitz-free
  spaces}, J. Funct. Anal., 280 (2021), pp.~Paper No. 108849, 38.

\bibitem{AG19}
{\sc R.~J. Aliaga and A.~J. Guirao}, {\em On the preserved extremal structure
  of {L}ipschitz-free spaces}, Studia Math., 245 (2019), pp.~1--14.

\bibitem{AP20}
{\sc R.~J. Aliaga and E.~Perneck\'{a}}, {\em Supports and extreme points in
  {L}ipschitz-free spaces}, Rev. Mat. Iberoam., 36 (2020), pp.~2073--2089.

\bibitem{APPP20}
{\sc R.~J. Aliaga, E.~Perneck\'{a}, C.~Petitjean, and A.~Proch\'{a}zka}, {\em
  Supports in {L}ipschitz-free spaces and applications to extremal structure},
  J. Math. Anal. Appl., 489 (2020), pp.~124128, 14.

\bibitem{AFGR}
{\sc L.~Antunes, V.~Ferenczi, S.~Grivaux, and C.~Rosendal}, {\em Light groups
  of isomorphisms of {B}anach spaces and invariant {LUR} renormings}, Pacific
  J. Math., 301 (2019), pp.~31--54.

\bibitem{BFS18}
{\sc Z.~M. Balogh, K.~F\"{a}ssler, and H.~Sobrino}, {\em Isometric embeddings
  into {H}eisenberg groups}, Geom. Dedicata, 195 (2018), pp.~163--192.

\bibitem{BTV-Carnot}
{\sc Z.~M. Balogh, T.~Titkos, and D.~Virosztek}, {\em Isometric rigidity of the
  {W}asserstein space $\mathcal{W}_1(\mathbf{G})$ over {C}arnot groups}, 2023.

\bibitem{Banachbook}
{\sc S.~Banach}, {\em Th{\'e}orie des op{\'e}rations lin{\'e}aires}, vol.~1 of
  Monogr. Mat., Warszawa, PWN - Panstwowe Wydawnictwo Naukowe, Warszawa, 1932.

\bibitem{BertrandKloeckner-Hadamard}
{\sc J.~Bertrand and B.~Kloeckner}, {\em A geometric study of {W}asserstein
  spaces: {H}adamard spaces}, J. Topol. Anal., 4 (2012), pp.~515--542.

\bibitem{Wreathbook}
{\sc M.~Bhattacharjee, D.~Macpherson, R.~G. M\"{o}ller, and P.~M. Neumann},
  {\em Notes on infinite permutation groups}, vol.~12 of Texts and Readings in
  Mathematics, Hindustan Book Agency, New Delhi; co-published by
  Springer-Verlag, Berlin, 1997.
\newblock Lecture Notes in Mathematics, 1698.

\bibitem{Graphbook}
{\sc J.~A. Bondy and U.~S.~R. Murty}, {\em Graph theory with applications},
  American Elsevier Publishing Co., Inc., New York, 1976.

\bibitem{CarnotBook}
{\sc A.~Bonfiglioli, E.~Lanconelli, and F.~Uguzzoni}, {\em Stratified {L}ie
  groups and potential theory for their sub-{L}aplacians}, Springer Monographs
  in Mathematics, Springer, Berlin, 2007.

\bibitem{CuthJohanis}
{\sc M.~C\'{u}th and M.~Johanis}, {\em Isometric embedding of {$\ell_1$} into
  {L}ipschitz-free spaces and {$\ell_\infty$} into their duals}, Proc. Amer.
  Math. Soc., 145 (2017), pp.~3409--3421.

\bibitem{FleJa1}
{\sc R.~J. Fleming and J.~E. Jamison}, {\em Isometries on {B}anach spaces:
  function spaces}, vol.~129 of Chapman \& Hall/CRC Monographs and Surveys in
  Pure and Applied Mathematics, Chapman \& Hall/CRC, Boca Raton, FL, 2003.

\bibitem{FleJa2}
\leavevmode\vrule height 2pt depth -1.6pt width 23pt, {\em Isometries on
  {B}anach spaces. {V}ol. 2}, vol.~138 of Chapman \& Hall/CRC Monographs and
  Surveys in Pure and Applied Mathematics, Chapman \& Hall/CRC, Boca Raton, FL,
  2008.
\newblock Vector-valued function spaces.

\bibitem{GLPPRZ18}
{\sc L.~Garc\'{\i}a-Lirola, C.~Petitjean, A.~Proch\'{a}zka, and A.~Rueda~Zoca},
  {\em Extremal structure and duality of {L}ipschitz free spaces}, Mediterr. J.
  Math., 15 (2018), pp.~Paper No. 69, 23.

\bibitem{GTV-Hilbert}
{\sc G.~P. Geh{\'e}r, T.~Titkos, and D.~Virosztek}, {\em The isometry group of
  {W}asserstein spaces: the {Hilbertian} case}, Journal of the London
  Mathematical Society,  (2022).

\bibitem{HilNeeb}
{\sc J.~Hilgert and K.-H. Neeb}, {\em Structure and geometry of {Lie} groups},
  Springer Monogr. Math., Berlin: Springer, 2012.

\bibitem{Kloe10}
{\sc B.~Kloeckner}, {\em A geometric study of {Wasserstein} spaces: {Euclidean}
  spaces}, Ann. Sc. Norm. Super. Pisa, Cl. Sci. (5), 9 (2010), pp.~297--323.

\bibitem{DS19}
{\sc E.~Le~Donne and S.~Rigot}, {\em Besicovitch covering property on graded
  groups and applications to measure differentiation}, J. Reine Angew. Math.,
  750 (2019), pp.~241--297.

\bibitem{Lu-thesis}
{\sc A.~Lukyanenko}, {\em Geometric mapping theory of the Heisenberg group,
  sub-Riemannian manifolds, and hyperbolic spaces}, University of Illinois at
  Urbana-Champaign, Ph.D. thesis, 2014.

\bibitem{M-W81}
{\sc E.~Mayer-Wolf}, {\em Isometries between {B}anach spaces of {L}ipschitz
  functions}, Israel J. Math., 38 (1981), pp.~58--74.

\bibitem{Meg01}
{\sc M.~G. Megrelishvili}, {\em Operator topologies and reflexive
  representability}, in Nuclear groups and {L}ie groups ({M}adrid, 1999),
  vol.~24 of Res. Exp. Math., Heldermann, Lemgo, 2001, pp.~197--208.

\bibitem{MySt}
{\sc S.~B. Myers and N.~E. Steenrod}, {\em The group of isometries of a
  {R}iemannian manifold}, Ann. of Math. (2), 40 (1939), pp.~400--416.

\bibitem{OMatroidBook}
{\sc J.~Oxley}, {\em Matroid theory}, vol.~21 of Oxford Graduate Texts in
  Mathematics, Oxford University Press, Oxford, second~ed., 2011.

\bibitem{TanYang}
{\sc K.-H. Tan and X.-P. Yang}, {\em Characterisation of the sub-{R}iemannian
  isometry groups of {$H$}-type groups}, Bull. Austral. Math. Soc., 70 (2004),
  pp.~87--100.

\bibitem{V23}
{\sc T.~Veeorg}, {\em Characterizations of {D}augavet points and delta-points
  in {L}ipschitz-free spaces}, Studia Math., 268 (2023), pp.~213--233.

\bibitem{Weaverbook}
{\sc N.~Weaver}, {\em Lipschitz algebras}, World Scientific Publishing Co. Pte.
  Ltd., Hackensack, NJ, second~ed., 2018.

\bibitem{Weaver18}
\leavevmode\vrule height 2pt depth -1.6pt width 23pt, {\em On the unique
  predual problem for {L}ipschitz spaces}, Math. Proc. Cambridge Philos. Soc.,
  165 (2018), pp.~467--473.

\end{thebibliography}
\end{document}